\documentclass[letterpaper,11pt]{article}
\usepackage{amsmath,amsfonts,amssymb,latexsym}
\usepackage{fullpage}
\usepackage{amsmath}
\usepackage{amsthm}
\usepackage{mathtools}
\usepackage{parskip}
\usepackage{graphicx}
\usepackage{placeins}
\usepackage{xcolor}
\usepackage{hyperref}
\usepackage{subcaption}
\usepackage{thmtools}
\usepackage{thm-restate}
\usepackage{dsfont}
\usepackage{placeins}
\usepackage{wrapfig}
\usepackage{todonotes}
\usepackage{enumerate}
\usepackage{float}
\usepackage{xcolor}
\usepackage{wasysym}
\usepackage{hyperref}
\usepackage{cleveref}
\usepackage{sectsty}
\usepackage{braket}
\sectionfont{\fontsize{12}{15}\selectfont}
\usepackage{caption}
\definecolor{celadon}{rgb}{0.17, 0.8, 0.69}
\definecolor{darkred}{rgb}{.7,0,0}
\definecolor{darkpink}{rgb}{0.95, 0.1, 0.8}

\declaretheorem[numberwithin = section]{theorem}
\declaretheorem[sibling = theorem]{lemma}
\declaretheorem[sibling = theorem]{corollary}
\declaretheorem[sibling = theorem]{definition}

\declaretheorem[sibling = theorem]{assumptions}
\declaretheorem[sibling = theorem]{remark}

\DeclareMathOperator*{\esssup}{ess\,sup}
\DeclareMathOperator*{\essinf}{ess\,inf}

\def\bbR{\mathbb{R}}
\def\bbC{\mathbb{C}}
\def\p{\partial}
\def\wh{\widehat}
\def\eps{\epsilon}

\def\R{\mathbb{R}}

\def\mcF{\mathcal{F}}
\def\mcG{\mathcal{G}}
\def\mcT{\mathcal{T}}

\def\la{\langle}
\def\ra{\rangle}

\renewcommand{\Omega}{\mathcal{D}}

\newcommand{\TT}{\mathcal{T}}
\newcommand{\TC}{\overline{\TT}}

\numberwithin{equation}{section}

\hypersetup{
    colorlinks=true,
    linkcolor=blue,
    filecolor=magenta,      
    urlcolor=blue,
    pdftitle={Learning Markovian Homogenized Models in Viscoelasticity},
    pdfpagemode=FullScreen,
    }

\begin{document}

\noindent 
\begin{minipage}{6.4in}
  \medskip
  \bigskip
  \begin{center}
    {\Large Learning Markovian Homogenized Models in Viscoelasticity} \\[3mm]
  \end{center}
\end{minipage}
\medskip
\begin{center}
\textbf{Kaushik Bhattacharya\footnote{Mechanical and Civil Engineering, California Institute of Technology, Pasadena, CA, USA (bhatta@caltech.edu)}, Burigede Liu\footnote{Department of Engineering, University of Cambridge, Cambridge, UK (bl377@eng.cam.ac.uk)}, Andrew Stuart\footnote{Computing and Mathematical Sciences, California Institute of Technology, Pasadena, CA, USA \\ \indent\indent (astuart@caltech.edu)}, Margaret Trautner\footnote{Computing and Mathematical Sciences, California Institute of Technology, Pasadena, CA, USA \\ \indent\indent(trautner@caltech.edu)}}
\end{center}
\medskip

\textbf{Abstract: } Fully resolving dynamics of materials with rapidly-varying features involves expensive fine-scale computations which need to be conducted on macroscopic scales. The theory of homogenization provides an approach to derive effective macroscopic equations which eliminates the small scales by exploiting scale separation. An accurate homogenized model avoids the computationally-expensive task of numerically solving the underlying balance laws at a fine scale, thereby rendering a numerical solution of the balance laws more computationally tractable. 

In complex settings, homogenization only defines the constitutive model implicitly, and machine learning can be used to learn the constitutive model explicitly from localized fine-scale simulations. In the case of one-dimensional viscoelasticity, the linearity of the model allows for a complete analysis. We establish that the homogenized constitutive model
may be approximated by a recurrent neural network (RNN) that captures the memory. The memory is encapsulated in the evolution of an appropriate finite set of internal variables, discovered through the learning process and dependent on the history of the strain. Simulations are presented which validate the theory. Guidance for the
learning of more complex models, such as arise in plasticity, by similar techniques,
is given.

\bigskip
\section{Introduction}
\label{sec:I}

Many problems in continuum mechanics lead to constitutive laws which are
history-dependent. This property may be inherent to  physics beneath the
continuum scale (for example
in plasticity \cite{simo1998numerical,simo2006computational}) or may arise from homogenization of rapdily varying continua \cite{bensoussan2011asymptotic,pavliotis2008multiscale}  
(for example in the Kelvin-Voigt (KV) model of viscoelasticity \cite{francfort1986homogenization}). When history-dependence
is present, Markovian models that capture the history dependence
are desirable for both interpretability
and computability. In some cases theory may be used
to justify Markovian models which capture this history-dependence,
but in many cases data plays a central role in finding such models.
The goal of this paper is to study data-driven methods to learn
Markovian models for history-dependence and to provide theoretical
underpinnings for understanding them.

The paper \cite{liu2022learning} adopted a data-driven learning
approach to uncovering history-dependent homogenized models arising
in crystal plasticity. However, the resulting constitutive model is
not causal and instead learns causality approximately from computations
performed at the level of the cell problem.  The paper \cite{biswriting} 
introduces a different approach, learning \emph{causal} constitutive 
models of plasticity. In order to give rigorous underpinnings to the empirical results therein, 
the present work is devoted to studying
the methodology from \cite{biswriting} in the setting of linear one-dimensional viscoelasticity. 
Here we can use theoretical understanding to justify 
and validate the methodology; we show that machine-learned homogenized models 
can accurately approximate the dynamics of multiscale models at much cheaper evaluation cost. We obtain insight into desirable choice of training data to learn the homogenized constitutive model, and we study the effect of
the multiple local minimizers which appear in the underlying optimization problem.
Furthermore, the rigorous underpinnings 
enable us to gain insight into how to test model hypotheses. We demonstrate 
that hypothesizing the correct model leads to robustness with respect to changes in time-discretization in the causal model: 
 the model can be trained at one time-step and
used at others; and the model can be trained with
one time-integration method and used with others.
In contrast, hypothesizing an incorrect 
model leads to intolerable sensitivity with respect  to the  time-step. Thus training
at one time-step and testing at other levels of resolution provides
a method for testing model form hypotheses.
We work primarily with the one-dimensional KV model for viscoelasticity for which 
the constitutive model depends only on strain and strain rate. We will also 
briefly touch on the standard linear solid (SLS) model for which 
the constitutive relation depends only on the strain and the strain history;
in so doing we show that the ideas presented extend beyond the specifics 
of the one-dimensional KV setting.

In Subsection \ref{ssec:SU} we describe the overarching mathematical framework
adopted, and Subsection \ref{ssec:LR} contains a detailed literature review. This is followed, in Subsection \ref{ssec:OC}, by a statement of our contributions and an overview of the paper. 
In Subsection \ref{ssec:N} we summarize  notation used throughout the
remainder of the paper.

\subsection{Set-up}
\label{ssec:SU}
Consider the problem of material response on an arbitrary spatial domain $\Omega \subset \bbR^d$ where the material properties vary rapidly within the domain. 
We denote by $u_{\epsilon} \in \bbR^d$ the displacement, where  $\epsilon: 0 < \epsilon \ll 1$ denotes  the  scale of the material fluctuations.
Denote by $\TT = (0,T)$ the time domain of interest. We consider continuum models which satisfy dynamical equations of the form
\begin{subequations}\label{eqns:general-dynamics}\begin{align}
   \rho\p_t^2u_{\epsilon}&=\nabla\cdot \sigma_{\epsilon}  + f, \quad  (x,t) \in\Omega \times \TT, \label{eqn:force-balance} \\
    \sigma_{\epsilon}(x,t)  &= \Psi^{\dagger}_{\epsilon}\left(\nabla u_{\epsilon}(x,t), \p_t \nabla u_{\epsilon}(x,t), \left\{\nabla u_{\epsilon}(x,\tau)\right\}_{\tau \in \TC},x,t\right),
    \quad  (x,t) \in\Omega \times \TT, \label{eqn:const-gen}\\
    u_{\epsilon}  &= u^*, \quad \p_t u_{\epsilon}  = v^*, \quad  (x,t) \in\Omega \times \{0\}, \\
    u_{\epsilon}  &= 0, \quad  x  \in \p\Omega, \quad  (x,t) \in\partial \Omega \times \TT.
\end{align}
\end{subequations}
From these equations we seek $u_{\epsilon}: \Omega \times \TT \mapsto \bbR^d$.
Equation \eqref{eqn:force-balance} is the balance equation with inertia
term $\rho\p_t^2 u_{\epsilon}$ for known parameter $\rho \in \bbR^+$, resultant stress term $\nabla\cdot(\sigma_{\epsilon})$ where $\sigma_{\epsilon} \in \bbR^{d \times d}$ is the internal stress tensor, and known external forcing $f \in \bbR^d$; equations (\ref{eqns:general-dynamics}c,\ref{eqns:general-dynamics}d) specify the
initial and boundary data for the displacement. Equation \eqref{eqn:const-gen} is the constitutive law relating properties of the the strain $\nabla u_{\epsilon}$ to the stress $\sigma_{\epsilon}$ via map $\Psi^\dagger_{\epsilon}$. In this paper we will consider this model
with inertia ($\rho>0$) and without inertia ($\rho \equiv 0$).

\subsubsection{Constitutive Model}
\label{sssec:CM}

The constitutive model is defined by
$\Psi^{\dagger}_{\epsilon}:\R^{d\times d} \times \R^{d\times d} \times
C(\TC; \R^{d\times d}) \times \Omega \times \TT \to \R^{d \times d}.$
It takes time $t$ as input, which
enables the stress at time $t$ to be expressed only
in terms of the strain history up to time $t$, 
$\left\{\nabla u_{\epsilon}(x,\tau)\right\}_{\tau=0}^t$, and not on the
future of the strain for $\tau \in (t,T].$ It takes $x$ as input to allow
for material properties which depend on the rapidly varying $x/\epsilon;$ 
it is also possible to allow for material properties which exhibit 
additional dependence on the slowly varying $x$, but we exclude this case
for simplicity.

Such constitutive models include a variety of plastic, viscoelastic and viscoplastic materials. In this paper we focus mainly on the one-dimensional KV viscoelastic setting in order to highlight
ideas; in this case $\Psi^\dagger_{\epsilon}$ is independent of
the history of the strain. However, we will briefly demonstrate that
similar concepts relating to the learning of constitutive models
also apply to the case of an SLS,
for which $\Psi^\dagger_{\epsilon}$ is independent of the strain rate
but does depend on the history of the strain; the SLS
contains the KV and Maxwell models of viscoelasticity
as special cases. Furthermore, the paper \cite{biswriting} includes 
empirical evidence that similar concepts relating to the learning of constitutive models also apply in the setting of Maxwell models of
plasticity.


\subsubsection{Homogenized Constitutive Model}
\label{sssec:HCM}

The goal of homogenization is to find constitutive models which eliminate small-scale dependence. To this end, we first discuss the form of a general homogenized problem:
we seek the equation satisfied by $u_0$, an appropriate limit of $u_{\epsilon}$ as $\epsilon \to 0.$ Then map $\Psi^\dagger_0$ defines the constitutive relationship in a homogenized model for $u_0$ of the form
\begin{subequations}\label{eqns:general-dynamics-homog}\begin{align}
   \rho\p_t^2u_{0}&=\nabla\cdot \sigma_{0}  + f, \quad  (x,t) \in\Omega \times \TT, \label{eqn:force-balance-homog} \\
    \sigma_{0}(x,t)  &= \Psi^{\dagger}_{0}\left(\nabla u_{0}(x,t), \p_t \nabla u_{0}(x,t), \left\{\nabla u_{0}(x,\tau)\right\}_{\tau \in \TC},t\right),
    \quad  (x,t) \in\Omega \times \TT, \label{eqn:const-gen-homog}\\
    u_{0}  &= u^*, \quad \p_t u_{0}  = v^*, \quad  (x,t) \in\Omega \times \{0\}, \\
    u_{0}  &= 0, \quad   \quad  (x,t) \in \p \Omega \times \TT.
\end{align}
\end{subequations}
The key assumption about this homogenized model is that parameter $\epsilon$
no longer appears. Furthermore, since we assumed that the multiscale model
material properties depend only on the rapidly varying scale $x/\epsilon$
and not on $x$, we have that $\Psi^{\dagger}_{0}$ is independent of $x$:
$\Psi^{\dagger}_{0}:\R^{d\times d} \times \R^{d\times d} \times
C(\TC; \R^{d\times d}) \times \TT \to \R^{d \times d}.$

If the homogenized model is identified correctly then dynamics under the multiscale model $\Psi_{\epsilon}^\dagger$, i.e. $u_{\epsilon}$, can approximated by dynamics under the homogenized model $\Psi_0^\dagger$, i.e. $u_0$. 
This, potentially, facilitates cheaper computations since
length-scales of size $\epsilon$ need not be resolved.
We observe, however, that for KV
viscoelasticity, the homogenized model contains history-dependence (memory)
even though the multiscale model does not. Markovian history-dependence is desirable for two primary reasons: first, Markovian models
encode conceptual understanding, representing the history dependence
in a compact, interpretable form; second, Markovian expression reduces computational
cost from ${\mathcal O}(|\TT|^2)$ in the general memory case to
${\mathcal O}(|\TT|)$ in the Markovian case.
In the general media setting, for a multitude of models in viscoleasticity,
viscoplasticity, and plasticity, the homogenized model will depend on the memory in a non-Markovian manner. However, it is interesting to determine situations
in which accurate Markovian approximations can be found.

\subsubsection{Markovian Homogenized Constitutive Model}
\label{sssec:MHCM}

We will seek to identify \emph{internal variables} $\xi$ and functions
$\mathcal{F}, \mathcal{G}$ such that, for $B \in C(\TT;\R^{d \times d})$, $\Psi^{\dagger}_{0}$
can be approximated by
\begin{subequations}
\label{eq:MARKOV}
\begin{align}
\Psi_{0}\left(B(t), \p_t B(t), \left\{B(\tau)\right\}_{\tau \in \TC},t\right) &=\mathcal{F}\left(B(t),\p_t B(t),\xi(t)\right),\\
\p_t \xi(t)&=\mathcal{G}\left(\xi(t),B(t)\right).
\end{align}
\end{subequations}
Note that $\xi$ carries the history dependence on $B$ through its Markovian
evolution. We assume that $\xi \in C(\TT;\R^r)$ for some integer $r$ and 
hence that $\mathcal{F}:\R^{d \times d} \times \R^{d \times d} \times \R^r \to \R^{d \times d}$ and that $\mathcal{G}:\R^r \times \R^{d \times d} \to \R^r.$
In dimension $d>1$ there will be further symmetries that should be built
into the model, but as the concrete analysis in this paper is in
dimension $d=1$ we will not detail these symmetries here.

In general such a Markovian model can only {\em approximate} the true model,
and the nature of the physics leading to a good approximation will
depend on the specific continuum mechanics problem. 
To determine $\mathcal{F}$ and $\mathcal{G}$ in practice we will parameterize them as neural
networks, which enables us to use general purpose optimization software to
determine suitable values of the parameters. In doing so we
identify an operator class $\Psi_0\left(\cdot \; ; \theta \right)$ and parameter space $\Theta$ such that, for some judiciously chosen $\theta^*\in \Theta$, $\Psi_0\left(\cdot \; ; \theta^* \right)\approx \Psi_{\epsilon}^\dagger.$

In this paper
we will concentrate on justifying a Markovian homogenized approximation
in the context of one-dimensional KV viscoelasticity.
Our justification will use theory that is specific to one dimensional 
linear viscoelasticity, and we demonstrate that the approach also works for the general SLS,
which includes the KV model as a particular limit.
Furthemore, the paper \cite{biswriting} contains evidence 
that the ideas we develop apply
beyond the confines of one dimensional  linear viscoelasticity
and into nonlinear plasticity in higher spatial dimensions.

The specific property of one-dimensional viscoelasticity that
we exploit to underpin our analysis (and which applies
to the SLS and therefore also to the KV model) is that, for piecewise-constant media, the homogenized model has a memory term which can be
represented in a Markovian way. Therefore, to justify our strategy of
approximating by Markovian models we will: first, approximate the 
rapidly-varying medium by a piecewise-constant rapidly-varying medium;
secondly, homogenize this model to find a Markovian description; and 
finally, demonstrate how the Markovian description can be learned from 
data at the level of the unit cell problem, using neural networks.
For more general problems 
we anticipate a similar justification
holding, but with different specifics leading to the existence of
good approximate Markovian homogenized models. 
The benefit of the one-dimensional viscoelastic
setting is that, through theory, we obtain underpinning insight into the
conceptual approach more generally, and in particular for plasticity;
this theory underpins the numerical experiments which follow.

\subsection{Literature Review}
\label{ssec:LR}

The continuum assumption for physical materials approximates the inherently particulate nature of matter by a continuous medium and thus allows the use of partial differential equations to describe response dynamics.   We refer the reader to \cite{gurtin_cont_mech,spencer_cont_mech,gonzalez2008first} for a general background.  In continuum mechanics, the governing equations are derived by combining universal balance laws of physics (balance of mass, momenta and energy) with a constitutive relation that describes the properties of the material being studied.  This is typically specified as the relation between a dynamic quantity like stress or energy and kinematic quantities like strain and its history.  The constitutive relation of many materials are history dependent, i.e., the state of stress at an instant depends on the history of deformation.  It is common in continuum mechanics to incorporate this history dependence through the introduction of internal variables.  We refer the reader to  \cite{rice1971inelastic} for a systematic formulation of internal variable theories.

Of particular interest in this work are viscoelastic materials.  We refer the reader to  \cite{christensen_visco,lakes_visco} for a general background.  In viscoelastic materials, the state of stress at any instant depends on the strain and its history.  There are various models where the stress depends only on strain and strain rate (Kelvin-Voigt), internal variables (standard linear solids), convolution kernels and fractional time derivatives.

While constitutive relations were traditionally determined empirically, more recently there has been a systematic attempt to understand them from more fundamental solids, and this has given rise to a rich activity in multiscale modeling of materials \cite{fish2010multiscale,tadmor,van2020roadmap}.
Materials are heterogeneous on various length (and time) scales, and it is common to use different theories to describe the behavior at different scales \cite{phillips2001crystals}.  The goal of multiscale modeling of materials is to use this hierarchy of scales to understand the overall constitutive behavior at the scale of applications.  The hierarchy of scales include a number of continuum scales.  For example, a composite material is made of various distinct materials arranged at a scale that is small compared to the scale of application, but large enough compared to an atomistic/discrete scale that their behavior is adequately described by continuum mechanics.  Or, for example, a polycrystal, is made of a large collection of grains (regions  of identical anisotropic material but with differing orientation) that are small compared to the scale of application but large enough for a continuum theory.  Homogenization theory leverages the assumption of the separation of scales to average out the effects of fine-scale material variations.  To estimate macroscopic response of heterogeneous materials, asymptotic expansion of the displacement field yields a set of boundary value problems whose solution produces an approximation that does not depend on the microscale \cite{bensoussan2011asymptotic,sanchez1980non}. The fundamentals of asymptotic homogenization theory are well-established \cite{pavliotis2008multiscale,cioranescu_donato,allaire1992homogenization}.  
Milton \cite{milton_book} provides a comprehensive survey of the effective or homogenized properties.  

Homogenization in the context of viscoelasticity was initiated by Sanchez-Palencia (\cite{sanchez1980non}  Chapter 6), who pointed out that the homogenization of a Kelvin-Voigt model leads to a model with fading memory.  Further discussion of homogenization theory in (thermo-)viscoelasticity can be found in
Francfort and Suquet \cite{francfort1986homogenization}, and a detailed discussion of the overall behavior including memory in Brenner and Suquet \cite{brenner2013overall}.  A broader discussion of homogenization and memory effects can be found in Tartar \cite{tartar_1990}.  It is now understood that homogenization of various constitutive models gives rise to memory.

As noted above, according to homogenization theory, the macroscopic behavior depends on the solution of a boundary value problem at the microscale.  Evaluating the macroscopic behavior by the solution of a boundary value problem computationally leads to what has been called computational micromechanics \cite{zohdi_introduction_2005}.  These often involve periodic boundary conditions, and fast Fourier transform-based methods are widely used since Moulinec and Suquet \cite{moulinec1998numerical} (see \cite{mishra_review,moulinec_review} for recent summaries).  While these enable us to compute the macroscopic response for a particular deformation history, one needs to repeat the calculation for all possible deformation histories.

Therefore, recent work in the mechanics literature addresses the issue of learning 
homogenized constitutive models from  computational data
\cite{mozaffar2019deep,wu_rnn,liu2022learning,biswriting} or experimental data \cite{as2022mechanics}. This learning problem
requires determination of maps that take as inputs functions describing
microstructural properties, and leads in to the topic of operator learning.


Neural networks  on finite-dimensional spaces have proven effective at solving longstanding computational problems. A natural question which arises from this success is whether neural networks can be used to solve partial differential equations. From a theoretical standpoint, the classical notion of neural networks as maps between finite-dimensional spaces is insufficient. Indeed, a differential operator is a map between infinite-dimensional spaces. Similar to the case of finite-dimensional maps, universal approximation results for nonlinear operator learning have been developed \cite{chen1995universal}. In one approach to operator learning, model reduction methods are applied to the operator itself. In this setting, a low-dimensional approximation space is assumed, and the operator approximation is constructed via regression using the latent basis \cite{peherstorfer2016data,qian2020lift}. In a second method, classical dimension reduction maps are applied to the input and output spaces, and an approximation map is learned between these finite-dimensional spaces \cite{bhattacharya2020model}. An important result of this method is its mesh-invariance property. Data for any operator-learning problem must consist of a finite discretization of the true input and output functions. One side effect of this practical fact is that some existing methods for operator learning depend critically on the choice of discretization. Mesh-invariant methods are desirable for both practical and theoretical purposes. Practically, it would be expensive to train a new model to accomodate a finer data resolution. From a theory standpoint, since the true map between infinite-dimensional spaces is inherently resolution-independent, the operator approximation ought to have this property as well. Mesh-invariance allows the operator approximation to be applied in the use of various numerical approximation methods for PDEs, which is of particular importance when the operator is being used as a surrogate model for the overall PDE system \cite{li2020neural}.

Surrogate modeling bypasses expensive simulation computations by replacing part of the PDE with a neural network. In several application settings, including fluid flows and solid mechanics computations, surrogate modeling has met empirical success in approximating the true solution \cite{sun2020surrogate,haghighat2021physics}. This work continues to use ideas from physics-informed machine learning; in \cite{haghighat2021physics} in particular, the differential operator is incorporated into the cost function via automatic differentiation. Other work in surrogate modeling for solid mechanics proposes a hierarchical network architecture that mimics the heterogeneous material structure to yield an approximation to the homogenized solution \cite{liu2019deep}.

In this paper we use an RNN as a surrogate model for the constitutive relation on the microscale. The RNN can then be used to evaluate the forward dynamic response on the microscale cells, whose results are combined with traditional numerical approximation methods to yield the macroscale response. Furthermore, the RNN that we train at
a particular time-discretization is also accurate when used at other time-discretizations, if the correct model form is proposed. However, 
the RNN is trained for a particular choice of material parameters, and the resulting model cannot be used at different material parameter values; simultaneously
learning material parameter dependence is left for future work.

\subsection{Our Contributions and Paper Overview}
\label{ssec:OC}
Our contributions are as follows:
\begin{enumerate}
    \item We provide theoretical underpinnings for the discovery of 
    Markovian homogenized models in viscoelasticity and 
plasticity; the methods use data generated by solving cell problems to learn
constitutive models of the desired form.
    \item We prove that in the one-dimensional Kelvin-Voigt (KV) setting, any solution of the multiscale problem can be approximated by solution of a homogenized problem with Markovian structure.
    \item We prove that the constitutive model for this Markovian homogenized  system can be approximated by a recurrent neural network (RNN) operator,
    learnt from data generated by solving the appropriate cell problem.
    \item We provide simulations which numerically demonstrate the accuracy of
    the learned Markovian model.
    \item We offer guidance for the application of this methodology, beyond the
    setting of one-dimensional viscoelasticity, into multi-dimensional plasticity.
    \end{enumerate}

In Section \ref{sec:KV}, we formulate the KV viscoelastic problem and its homogenized solution. In Section \ref{sec:main_theorems}, we present our
main theoretical results, addressing contributions 1., 2. and 3.; 
these are in the setting of one-dimensional 
KV viscoelasticity. We prove that solution of the multiscale problem
can be approximated by solution of a homogenized Markovian memory-depedent 
model that does not depend on small scales, and we prove that an RNN can 
approximate the constitutive law for this homogenized problem. Section 
\ref{sec:Numerics} contains numerical experiments which address
contributions 4. and 5.; the start of that section details the findings.

\subsection{Notation}
\label{ssec:N}
We define notation that will be useful throughout the paper. Recall
that $\mathcal{T} = (0,T)$ is the time domain of interest, and in the
one-dimensional setting we let $\mathcal{D} = [0,L]$ be the spatial domain.
Let $\la \cdot, \cdot \ra$ and $\|\cdot\|$ denote the standard inner product and induced norm operations on the Hilbert space $L^2(\mathcal{D};\bbR)$. Additionally, let $\|\cdot \|_{\infty}$ denote the $L^{\infty}(\mathcal{D};\bbR)$ norm. 

It will also be convenient to define the $\xi-$dependent quadratic form
\begin{equation}\label{eqn:ipq}q_{\xi}(u,w) := \int_{\mathcal{D}} \xi(x) \frac{\p u(x)}{\p x}\frac{\p w(x)}{\p x}\; dx
\end{equation}
for arbitrary $\xi \in L^{\infty}\bigl(\mathcal{D};(0,\infty)\bigr)$; furthermore we define
\begin{equation}\label{eqn:xi_prop1}
    \xi^+ := \esssup_{x \in \mathcal{D}}\xi(x) < \infty
\end{equation}
and 
\begin{equation}\label{eqn:xi_prop2}
    \xi^- := \essinf_{x \in \mathcal{D}} \xi(x) \geq 0.
\end{equation} 
In this paper we always work with $\xi$ such that $\xi^- >0$. 
Under these assumptions $q_{\xi}(\cdot,\cdot)$ 
defines an inner-product, and we can define the following norm
$$\|u\|_{H_0^1, \xi}^2 := q_{\xi}(u,u)$$
from it; note also that we may define a norm on $H^1_0(\mathcal{D};\R)$ by
$$\|u\|_{H_0^1}^2 := q_{\mathds{1}}(u,u),$$
where $\mathds{1}(\cdot)$ is the function in $L^{\infty}\bigl(\mathcal{D};(0,\infty)\bigr)$ taking value $1$ in $D$ a.e.\,.
The resulting norms are all equivalent on the space $H^1_0(\mathcal{D};\bbR)$; this is a consequence of the
following lemma:

\begin{lemma} \label{lem:h01equi}
    For any $\xi_1, \xi_2 \in L^{\infty}\bigl(\mathcal{D};(0,\infty)\bigr)$ satisfying properties (\ref{eqn:xi_prop1}) and (\ref{eqn:xi_prop2}), the norms $\|u\|_{H_0^1,\xi_1}$ and $\|u\|_{H_0^1, \xi_2}$ are equivalent in the sense that $$\frac{\xi_2^-}{\xi_1^+} \|u\|_{H_0^1, \xi_1}^2 \leq \|u\|_{H_0^1, \xi_2}^2  \leq \frac{\xi_2^+}{\xi_1^-} \|u\|_{H_0^1, \xi_1}^2. $$
\end{lemma}
\begin{proof}
For $i = 1,2:$
\begin{equation*}
    \xi_i^- \int_0^1 \left|\frac{\p u }{\p x}\right|^2 dx \leq \int_0^1 \xi_i(x) \left|\frac{\p u}{\p x}\right|^2 dx \leq  \xi_i^+ \int_0^1 \left|\frac{\p u}{\p x}\right|^2 dx.
\end{equation*}
The result follows. 
\end{proof}

We denote by $V$ the Hilbert space $H_0^1(\mathcal{D};\bbR)$ noting
that, as a consequence of the preceding lemma, we may use
$q_{\xi}(u,w)$ as the inner product on this space for any
$\xi$ satisfying \eqref{eqn:xi_prop1} and \eqref{eqn:xi_prop2}. 
We also define $$ \mathcal{Z} = L^{\infty}(\mathcal{T};L^2(\mathcal{D};\bbR)),
\quad \mathcal{Z}_2 = L^{2}(\mathcal{T};L^2(\mathcal{D};\bbR))$$ 
with norms 
$$\|r\|_{\mathcal{Z}} = \esssup_{t \in \mathcal{T}}(\|r(\cdot,t)\|), \quad
\|r\|_{\mathcal{Z}_2} = \Bigl(\int_0^{\mathcal{T}}\|r(\cdot,t)\|^2 dt \Bigr)^{\frac12}.$$
We note that $\mathcal{Z}$ is continuously embedded into $\mathcal{Z}_2.$

\section{One-Dimensional Kelvin-Voigt Viscoelasticity}\label{sec:KV}

This paper is focused on one-dimensional KV viscoelasticity because 
the model is amenable to rigorous analysis. The resulting analysis 
sheds light on the learning of constitutive models more generally. 
In Section \ref{subsec:KV-eqns} we present the equations for the model 
in an informal fashion and in a weak form suitable for analysis; 
and in Section \ref{subsec:KV-homog} we homogenize the model and 
define the operator defining the effective constitutive model.

\subsection{Governing Equations and Weak Form} \label{subsec:KV-eqns}
The one-dimensional KV model for viscoelasticity postulates that stress is affine in the strain and strain rate, with affine transformation dependent on 
the (typically spatially varying) material properties. For a multiscale material varying with respect to $x/\epsilon$ we thus have the following
definition of $\Psi^{\dagger}_{\epsilon}$ from \eqref{eqns:general-dynamics},
in the one-dimensional KV model:
$$\sigma_{\epsilon} = E_{\epsilon} \p_x u_{\epsilon} + \nu_{\epsilon}\p^2_{xt}u_{\epsilon}$$ where $E_{\epsilon}(x) = E\left(\frac{x}{\epsilon}\right)$ and $\nu_{\epsilon}(x) = \nu\left(\frac{x}{\epsilon}\right)$ are rapidly-varying material elasticity and viscosity, respectively. Both $E$ and $\nu$ are assumed to be 1-periodic. Then equations \eqref{eqns:general-dynamics} without momentum ($\rho \equiv 0$) become 
\begin{subequations}\label{eqn:main1d}
\begin{align}
    -&\p_x\left(E_{\epsilon}\p_x u_{\epsilon} + \nu_{\epsilon}\p^2_{xt} u_{\epsilon}\right) = f, \quad (x,t) \in \p\Omega \times \TT,\label{eqn:main1d1}\\
    &u_{\epsilon}(0,x)  = u^*, \quad \p_t u_{\epsilon}(0,x)  = v^*, \quad x \in \p\Omega,\\
    &u_{\epsilon}(x,t)  = 0, \quad  x  \in \p\Omega.
\end{align}
\end{subequations}
Any classical solution to equations \eqref{eqn:main1d} will also solve the corresponding \textit{weak form}: find $u_{\epsilon} \in \mathcal{C}(\mathcal{T};V)$ such that 
\begin{equation}\label{eqn:gen_q}
    q_{\nu_{\epsilon}}(\p_t u_{\epsilon}, \varphi) + q_{E_{\epsilon}}(u_{\epsilon},\varphi) = \la f, \varphi \ra
\end{equation}
for all test functions $\varphi \in V$. 
\subsection{Homogenization}\label{subsec:KV-homog}
In the inertia-free setting $\rho=0$ we perform homogenization to eliminate 
the dependence on the small scale $\epsilon$ in \eqref{eqn:main1d}.
First, we take the Laplace transform of \eqref{eqn:main1d},  which gives, for Laplace parameter $s$
and with the hat symbol denoting Laplace transform,
\begin{align*}
    &-\p_x\left((E_{\eps}+\nu_{\eps}s)\p_x\wh{u}_{\eps}\right) = \wh{f},
    \quad x \in \Omega,\\
   &\wh{u}_{\epsilon}(x,s)  = 0, \quad  x  \in \p\Omega.
\end{align*}
The initial condition $u_{\epsilon}(0,x)  = u^*$ is applied
upon Laplace inversion. Since $\eps \ll 1$, we may apply standard techniques
from multiscale analysis \cite{bensoussan2011asymptotic,pavliotis2008multiscale} and seek a solution in the form
\begin{equation*}
    \widehat{u_{\epsilon}} = \widehat{u}_0 + \epsilon \widehat{u}_1 + \epsilon^2 \widehat{u}_2 + \dots\,.
\end{equation*}
For convenience, define 
$\wh{a}(s,y) = E(y) + \nu(y) s$. Note that $\wh{a}(s,\cdot)$ is $1$-periodic. The leading order term in our approximation, $\wh{u}_0$, solves the following uniformly elliptic PDE with Dirichlet boundary conditions:
\begin{subequations}\label{eqn:homogen}
    \begin{equation}
        -\p_x\left(\widehat{a}_0(s)\p_x \widehat{u}_0\right) = \widehat{f} \quad \text{for}\quad x \in \mathcal{D},
    \end{equation}
    \begin{equation}
        \widehat{u}_0 = 0 \quad \text{for} \quad  x\in \p\mathcal{D}.
    \end{equation}
\end{subequations}
Here the coefficient $\wh{a}_0$ is given by 
\begin{equation*}
    \wh{a}_0 = \int_0^1\left(\wh{a}(y) + \wh{a}(y)\p_y\chi(y)\right)\; dy
\end{equation*}
and $\chi(y): [0,1] \to \bbR$ satisfies the \textit{cell problem}
\begin{subequations} \label{eq:cell}
\begin{equation}
    -\p_y\left(\wh{a}(y)\p_y\chi(y)\right) = \p_y \wh{a}(y)
\end{equation}
\begin{equation}
\chi \text{ is $1$-periodic}, \quad \int_{0}^1\chi(y)dy = 0.
\end{equation}
\end{subequations}
Using this, the coefficient $\wh{a}_0$ can be computed explicitly as the harmonic average of the original coefficient $\wh{a}$ \cite{pavliotis2008multiscale}[Subsection 12.6.1]:
\begin{align}\label{eqn:a_0}
    \wh{a}_0(s) = \left\la \wh{a}(s)^{-1}\right\ra^{-1} = \left(\int_0^1 \frac{dy}{s\nu(y) + E(y)}\right)^{-1},
    \end{align}
where $\la \cdot \ra$ denotes spatial averaging over the unit cell. 

Equations \eqref{eqn:homogen} indicate that the homogenized 
map $\Psi_0^\dagger$ appearing in \eqref{eqns:general-dynamics-homog} is, 
for one-dimensional linear viscoelasticity, defined from 
\begin{equation}\label{eqn:Lap_inv}
\Psi^{\dagger}_{0}\left(\p_{x} u_{0}(x,t), \p_{xt} u_{0}(x,t), \left\{\p_{x} u_{0}(x,\tau)\right\}_{\tau \in \TC},t\right)=\mathcal{L}^{-1}\Bigl(\wh{a}_0(s)\p_x\wh{u}_0\Bigr);
\end{equation}
here $\mathcal{L}^{-1}$ denotes the inverse Laplace transform.
Note that \eqref{eqn:a_0} shows that $\wh{a}_0$ grows linearly in $s \to \infty$ and computing
the constant term in a regular power series expansion at $s=\infty$ shows that we may write
$$\wh{a}_0=\nu' s + E'+\wh{\kappa}(s)$$
where $\wh{\kappa}(s)$ decays to $0$ as $s \to \infty$.
Here 
$$\nu'=\Bigl(\int_0^1 \frac{1}{\nu(y)} dy \Bigr)^{-1}, \quad  E'=\Bigl(\int_0^1 \frac{E(y)}{\nu(y)^2} dy \Bigr)\Big/\Bigl(\int_0^1 \frac{1}{\nu(y)} dy \Bigr)^{2}.$$
Details is presented in Appendix \ref{assec:Enu_s_limit}. Laplace inversion of $\wh{a}_0(s)\p_x\wh{u}_0$ then yields the conclusion that
\begin{equation}\label{eqn:homogc_general}
    \Psi_0^{\dagger}(\p_x u_0(t), \p^2_{xt} u_0(t), \{\p_x u_0(\tau)\}_{\tau \in \overline{\mathcal{T}}}. t; \theta) = E' \p_x u_0(t) + \nu'\p^2_{xt} u_0(t)  
+ \int_0^t \kappa(t-\tau) \p_x u_0(\tau) \; d\tau.
\end{equation}

\begin{remark}\label{rem:inertia}
When $\rho = 0$, the homogenized solution provably approximates $u_{\eps}$ in the 
$\eps \to 0$ limit; see Theorem \ref{thm:FS}. However, although we derived it with inertia set to zero, 
the homogenized solution given by equation \eqref{eqns:1dVE_PDE} is also valid when the inertia term $\rho\p_t^2u_\epsilon$ generates contributions which are $\mathcal{O}(1)$ with respect to $\eps.$
\end{remark}

The homogenized PDE for one-dimensional viscoelasticity follows by combining equations \eqref{eqns:general-dynamics-homog} with equation \eqref{eqn:homogc_general} to give
\begin{subequations}\label{eqns:1dVE_PDE}\begin{align}
   \rho\p_t^2u_{0}&=\nabla\cdot \sigma_{0}  + f, \quad  (x,t) \in\Omega \times \TT, \label{eqn:force-balance-homog} \\
    \sigma_{0}(t)  &=  E' \p_x u_0(t) + \nu'\p^2_{xt} u_0(t)  
+ \int_0^t \kappa(t-\tau) \p_x u_0(\tau) \; d\tau
    \quad  (x,t) \in\Omega \times \TT, \label{eqn:const-gen-homog}\\
    u_{0}  &= u^*, \quad \p_t u_{0}  = v^*, \quad  (x,t) \in\Omega \times \{0\}, \\
    u_{0}  &= 0, \quad   \quad  (x,t) \in \p \Omega \times \TT.
\end{align}
\end{subequations}

The price paid for homogenization 
is dependence on the strain history.
We will show in the next section, however, that we can 
approximate the general homogenized map with one in 
which the history-dependence is expressed in a Markovian manner.

\section{Main Theorems: Statement and Interpretation}\label{sec:main_theorems}

In this section we present theoretical results of three types. 
Firstly, in Subsection \ref{subsec:PC-approx}, we show that the solution 
$u_{\epsilon}$ to equation \eqref{eqn:main1d} is Lipschitz when viewed
as a mapping from the unit cell material properties $E(\cdot), \nu(\cdot)$
in $L^\infty$ into $\mathcal{Z};$ hence, an $\mathcal{O}(\delta)$ approximation of $E,\nu$ by piecewise-constant functions leads to an $\mathcal{O}(\delta)$ approximation of $u_{\epsilon}.$ Second, in Subsection \ref{subsec:homog}, we demonstrate that the homogenized  model based on piecewise-constant material properties can be represented in a Markovian fashion by introducing \emph{internal
variables}; hence, combining with the first point, we have a mechanism to
approximate $u_{\epsilon}$ by solution  a Markovian homogenized model. 
Third, in Subsection \ref{subsec-RNN}, we show the existence of neural 
networks which provide arbitrarily good approximation 
of the constitutive law arising in the Markovian homogenized model; this
suggests a model class within which to learn homogenized, Markovian
constitutive models from data. Subsection \ref{ssec:RNN_Optimization} 
establishes our framework for the optimization methods used to learn such
constitutive models; this framework is employed in the 
subsequent Section \ref{sec:Numerics}.

\begin{assumptions}\label{ass:stab}
We will make the following assumptions on $E$, $\nu$, and $f$ throughout:
\begin{enumerate}
    \item $f \in L^2(\mathcal{D};\bbR)$ for all $t \in \TC $; thus $\|f\|_{\mathcal{Z}} < \infty;$
    \item $E^+, \nu^+ < \infty$, and $E^-, \nu^- > 0.$
\end{enumerate}
\end{assumptions}

Note that $E^+ = E_{\epsilon}^+$ and $\nu^+ = \nu_{\epsilon}^+$, 
so we will drop the $\epsilon$ superscript in this notation.

\subsection{Approximation by Piecewise Constant Material}\label{subsec:PC-approx}
Consider \eqref{eqn:main1d} with continuous material properties $E$ and $\nu$. We show in Theorem \ref{thm:PC_approx} that we can approximate the solution $u_{\eps}$ to this system by a solution $u_{\eps}^{PC}$ which solves \eqref{eqn:main1d} with suitable piecewise-constant  material properties $E^{PC}$ and $\nu^{PC}$, in such a way that $u_{\eps}$ and $u_{\eps}^{PC}$ are close. To this end we 
make  precise the definition of piecewise-constant material properties.

\begin{definition}[\bf{Piecewise Constant}]\label{def:piecewise} A material is {\em piecewise constant} on the unit cell with $L$ pieces if the elasticity function $E(y)$ and the viscosity function $\nu(y)$ both take constant values on $L$ intervals $[0,a_1),[a_1,a_2),\dots,[a_{L-1},1]$. In particular, $E(y)$ and $\nu(y)$ have discontinuities only at the same $L-1$ points in the unit cell. We use
terminology $L-$piecewise constant to specify the number of pieces.
\end{definition}
\begin{remark} The situation in which $E(y)$ and $\nu(y)$  have discontinuities at different values of $y \in (0,1)$ can be reduced to the case in Definition \ref{def:piecewise} by increasing the value of $L$.
\end{remark}

\begin{restatable}[{\bf Piecewise-Constant Approximation}]{theorem}{PCapprox}\label{thm:PC_approx}
Let $E$ and $\nu$ be piecewise continuous functions, with a finite number of discontinuities,  satisfying Assumptions \ref{ass:stab}; let $u_{\epsilon}$ be the corresponding solution to \eqref{eqn:main1d}.
Then, for any $\delta > 0$, there exist
piecewise constant $E^{PC}$ and $\nu^{PC}$ (in the sense of
Definition \ref{def:piecewise}) such that solution $u^{PC}_{\epsilon}$ 
of equations \eqref{eqn:main1d} with these material properties satisfies
$$ \|u^{PC}_{\epsilon} - u_{\epsilon}\|_{\mathcal{Z}} < \delta. $$
\end{restatable}

Note that Theorem \ref{thm:PC_approx} is stated in the setting of no inertia.
The proof depends on the following lemma; proof of both the theorem
and the lemma may be found in Appendix \ref{sapdx:PC_approx}.
We observe that, since the Lipschitz result is in the $L^{\infty}-$norm with respect to
the material properties, it holds with constant $C$ independent of $\epsilon$, in the
case of interest where the material properties vary rapidly on scale $\epsilon.$
\begin{restatable}[Lipschitz Solution]{lemma}{lemmau}\label{lem:u1u2}
Let $u_i$ be the solution to 
\begin{align}
    -&\p_x\left(E_{i}\p_x u_{i} + \nu_{i}\p^2_{xt} u_{i}\right) = f, \quad (x,t) \in \p\Omega \times \TT, \\
    &u_{i}(x,t)  = u^*, \quad (x,t) \in \mathcal{D} \times \{0\},\\
    &u_{i}(x,t)  = 0, \quad  (x,t)  \in \p\Omega \times \mathcal{T} ,
\end{align}
associated with material properties $E_i$, $\nu_i$, for $i \in \{1,2\}$,
and forcing $f$, all satisfying the Assumptions \ref{ass:stab}. Then 
$$ \|u_1 - u_2\|_{\mathcal{Z}} \leq C \left( \|\nu_1 - \nu_2 \|_{{\infty}} + \|E_1 - E_2 \|_{{\infty}}\right)$$
for some constant $C \in \bbR^+$ dependent on $f, E_i^+, E_i^-, \nu_i^+$, $\nu_i^-$ and $L$ and independent of $\epsilon.$
\end{restatable}

\subsection{Homogenization for Piecewise Constant Material}\label{subsec:homog}
We show in Theorem \ref{thm:PC_exact} that for piecewise-constant material properties $E(\cdot)$ and $\nu(\cdot)$, the homogenized map $\Psi_0^{\dagger}$ given in \eqref{eqn:homogc_general}
can be written explicitly with a finite number of parameters, and in particular the memory is expressible in a Markovian form. This Markovian form  implicitly defines a finite number of internal variables.

\begin{restatable}[\bf{Existence of Exact Parametrization}]{theorem}{exactpieceparamthm}\label{thm:PC_exact}
Let $\Psi_0^{\dagger}$ be the map from strain history to stress in the homogenized model, as defined by equation \eqref{eqn:homogc_general}, in a piecewise constant material with $L'+1$ pieces. Define $\Psi_0^{PC}: \R^2 \times
C(\TC; \R) \times \TT \times \Theta \to \R$ by
\begin{subequations}\label{eqn:PC_exact}
\begin{align}
    &\Psi_0^{PC}(\p_x u_0(t), \p^2_{xt} u_0(t), \{\p_x u_0(\tau)\}_{\tau \in \overline{\mathcal{T}}}, t; \theta) = E_0 \p_x u_0(t) + \nu_0\p^2_{xt} u_0(t)  - \sum_{\ell=1}^{L_0} \xi_{\ell}(t),\\
    &\p_t \xi_{\ell}(t)  = \beta_{\ell}\p_x u_0(t) - \alpha_{\ell}\xi_{\ell}(t),\, \xi_{\ell}(0)=0, \quad \ell \in \{1, \dots, L_0\},
\end{align}
\end{subequations}
with parameter space 
\begin{equation}\label{eqn:Theta}
        \Theta = \left(E_0 \in \bbR_+, \; \nu_0 \in \bbR_+,\; L_0 \in \mathbb{Z}_+, \; \alpha_0 \in \bbR^{L_0}_+, \; \beta_0 \in \bbR^{L_0}\right).
\end{equation}
Then, under Assumptions \ref{ass:stab}, there exists $\theta^* \in \Theta$ with $(E_0,\nu_0,L_0,\alpha_0,\beta_0) = (E',\nu', L',\alpha,\beta)$ such that $$\Psi_0^{\dagger}(\p_x u_0(t),\p^2_{xt} u_0(t), \left\{\p_x u_{0}(\tau)\right\}_{\tau \in \overline{\mathcal{T}}},t)=\Psi^{PC}_0(\p_x u_0(t),\p^2_{xt}u_0(t),\{\p_x u_0(\tau)\}_{\tau \in \overline{\mathcal{T}}},t; \theta^*)$$ for all $u_0 \in \mathcal{C}^2( \overline{\mathcal{D}}\times \TC; \bbR)$
and $t \in \mathcal{T}$. 
\end{restatable}
The proof of the above theorem may be found in Appendix \ref{sapdx:PC_exact}. 
Note that the model in equations \eqref{eqn:PC_exact} is Markovian. Furthermore, although the model in \eqref{eqn:PC_exact} requires an input of $t$ for evaluation, the spatial variable $x$ only enters implicitly through the local values of $\p_x u_0$ and $\p^2_{xt}u_0$; the model acts pointwise in space. Thus we have not included $x$ explicitly in the theorem statement,
for economy of notation. In what follows it is useful to define
$u_0^{PC}$ to be the solution to the following system defined 
with constitutive model $\Psi_0^{PC}$: 
\begin{subequations}\label{eqns:PCC}\begin{align}
   &\rho\p_t^2u_{0}^{PC}-\p_x\sigma_{0}  = f, \quad  (x,t) \in\Omega  \times \mathcal{T}\label{eqn:PCC-force-balance}\\
    &\sigma_0(t) = \Psi_0^{PC}\left(\p_xu^{PC}_{0}(t), \p_{xt}^2  u^{PC}_{0}(t), \left\{\p_x u^{PC}_{0}(\tau)\right\}_{\tau \in \overline{\mathcal{T}}},t\right), \quad (x,t) \in \Omega \times \mathcal{T}\\
    &u^{PC}_{0}\rvert_{t=0}  = u^*, \quad \p_t u^{PC}_{0}\rvert_{t=0}  = v^*, \quad (x,t) \in \Omega \times \{0\} \\
    &u^{PC}_{0}  = 0, \quad  (x,t) \in \p\Omega \times \mathcal{T}.
\end{align}
\end{subequations}

Using a homogenization theorem, together with approximation by piecewise
constant material properties, we now show that $u_{\epsilon}$ can be approximated
by $u_0^{PC}$; this will follow from the inequality
$$\|u_{\epsilon} -  u_0^{PC}\|_{\mathcal{Z}_2} \le 
\|u_{\epsilon}-u_{\eps}^{PC}\|_{\mathcal{Z}_2}
+\|u_{\eps}^{PC}-u_{0}^{PC}\|_{\mathcal{Z}_2}.$$
The first term on the righthand-side may be controlled using Theorem \ref{thm:PC_approx}.
The fact that dynamics under constitutive law
$\Psi^\dagger_{\eps}$ converge to those under 
$\Psi^\dagger_{0}$ as $\eps \to 0$ may be used to control the second term; 
this fact is a consequence of the following theorem:
\begin{theorem}\label{thm:FS} Under Assumptions \ref{ass:stab}, the solution $u_{\epsilon}$ to equations \eqref{eqn:main1d} converges weakly to $u_0$, the solution to equations \eqref{eqns:1dVE_PDE} with $\rho=0$, in $W^{1,2}(\mathcal{T};V)$. Thus, for any $\eta >0$ there exists $\epsilon_{\text{crit}} >0 $ such that for all $\eps \in (0,\eps_{\text{crit}})$,
\begin{equation}\label{eqn:FS}
\|u_{\epsilon} - u_0\|_{\mathcal{Z}_2} < \eta .
\end{equation}
\end{theorem}

\begin{proof}
Since $f \in \mathcal{Z}$,  continuous embedding gives
$f \in \mathcal{Z}_2.$ Applying Theorem 3.1\cite{francfort1986homogenization} 
(noting that the work in that paper is set in dimension $d=3$,
but is readily extended to dimension $d=1$)
establishes weak convergence of $u_{\epsilon}$ to $u_0$ in
$W^{1,2}(\TT,V)$. Hence strong convergence in $\mathcal{Z}_2$ follows,
by compact embedding of $W^{1,2}(\TT;V)$ into $\mathcal{Z}_2$.
\end{proof}

The following corollary is a consequence of Theorem \ref{thm:FS}.

\begin{corollary}
Under Assumptions \ref{ass:stab} and assuming $E,\nu$ are piecewise
constant, the solution $u_{\epsilon}^{PC}$ to equations \eqref{eqn:main1d} converges weakly to $u_0^{PC}$, the solution to equations \eqref{eqns:PCC} with $\rho=0$, in $W^{1,2}(\mathcal{T};V)$. Thus, for any $\eta >0$ there exists $\epsilon_{\text{crit}} >0 $ such that for all $\eps \in (0,\eps_{\text{crit}})$,
\begin{equation}\label{eqn:FS_PC}
    \|u_{\eps}^{PC} - u_0^{PC} \|_{\mathcal{Z}_2} < \eta.
\end{equation}
\end{corollary}

Combining this result with that of Theorem \ref{thm:PC_approx}, noting continuous
embedding of $\mathcal{Z}$ into $\mathcal{Z}_2$, allows us to approximate $u_{\eps}$ by $u_0^{PC}$:

\begin{corollary}\label{cor:ue_uPC}
Let $E$ and $\nu$ be piecewise-continuous functions, with a finite number of discontinuities satisfying, along with $f$, Assumptions \ref{ass:stab}; let $u_{\epsilon}$ be the corresponding solution to \eqref{eqn:main1d}. 
Then for any $\eta > 0$, there exists $L_{crit}$ and $\eps_{crit}$ with the property that for all $L\geq L_{crit}$ there are $L-$piecewise-constant $E^{PC}$ and $\nu^{PC}$  such that for all $\eps \in (0,\eps_{crit})$,   
the solution to $u_0^{PC}$ to \eqref{eqns:PCC} with $\rho=0$ satisfies
\begin{equation}\label{eqn:u_eps_u0_PC}
\|u_{\epsilon} -  u_0^{PC}\|_{\mathcal{Z}_2} < \eta.
\end{equation}
\end{corollary}

\subsection{Neural Network Approximation of Constitutive Model}\label{subsec-RNN}

For the specific KV model in dimension $d=1$
we know the postulated form of $\Psi_0^{PC}$
and can in principle use this directly as a constitutive model.
However, in more complex problems we do not know the constitutive model
analytically, and it is then desirable to learn it from data from within an expressive model class. To this end we demonstrate
that $\Psi_0^{PC}$ can be approximated by an operator $\Psi_0^{RNN}$ which
has a similar form to that defined by equations
\eqref{eqn:PC_exact} but in which the right hand sides
of those equations are represented by neural networks, leading
to a recurrent neural network structure.

Recall the definitions of $(E',\nu',L')$ and $\theta^*$ in Theorem \ref{thm:PC_exact}.
We first define the linear functions $\mathcal{F}^{PC}: \bbR \times \bbR \times \bbR^{L'} \to  \bbR$
and $\mathcal{G}^{PC}: \bbR^{L'} \times \bbR \to \bbR$ by
\begin{subequations}\label{eqns:alphaellbetaell}
\begin{align}
    \mathcal{F}^{PC}\left(b,c,r\right) &= E'b + \nu' c -  \la \mathds{1},r\ra\\
 \mathcal{G}^{PC}(r,b) &= -Ar + \beta b
\end{align}
\end{subequations}
where $A = \text{diag}(\alpha_{\ell}) \in \bbR^{L' \times L'}$ and $\beta = \{\beta_1,\dots,\beta_{\ell}\} \in \bbR^{L'}$. We then have
\begin{subequations}
 \label{eqn:psiPC}   
\begin{align}
    \Psi_0^{PC}(\p_x u_0(t),\p_{xt}^2 u_0(t),\{\p_x u_0(\tau)\}_{\tau \in \overline{\mathcal{T}}},t;\theta^*)  &= \mathcal{F}^{PC}\left(\p_x u_0(t),\p^2_{xt} u_0(t),\xi(t)\right),\\
    \dot{\xi}(t)  &= \mathcal{G}^{PC}(\xi(t),\p_x u_0(t)), \quad \xi(0)=0
\end{align}
\end{subequations}
as in Theorem \ref{thm:PC_exact}.

We seek to approximate this map by $\Psi_0^{RNN}$ defined
by replacing the linear functions $\mathcal{F}^{PC}$ and $\mathcal{G}^{PC}$ by neural
networks $\mathcal{F}^{RNN}: \bbR \times \bbR \times \bbR^{L'} \to  \bbR$
and $\mathcal{G}^{RNN}: \bbR^{L'} \times \bbR \to \bbR$ to obtain
\begin{subequations}
 \label{eqn:psiRNN}   
\begin{align}
    \Psi_0^{RNN}(\p_x u_0(t),\p_{xt}^2 u_0(t),\{\p_x u_0(\tau)\}_{\tau \in \overline{\mathcal{T}}},t) & = \mathcal{F}^{RNN}\left(\p_x u_0(t),\p^2_{xt} u_0(t),\xi(t)\right)\\
    \dot{\xi}(t) & = \mathcal{G}^{RNN}(\xi(t),\p_x u_0(t)), \quad \xi(0)=0.
\end{align}
\end{subequations}

Let $R>0$ and define the bounded set $\mathsf{Z}_R=\{w: \bbR^+ \to \bbR\; \rvert\; \sup_{t \in \mcT} |w(t)| \le R\}.$

\begin{restatable}[RNN Approximation]{theorem}{RNNapprox}\label{thm:RNN_approx}Consider $\Psi_0^{PC}$ defined as by equations \eqref{eqns:alphaellbetaell}, \eqref{eqn:psiPC}. Assume that
there exist $\rho>0$ and $0 \le B <\infty$ such that $\rho < \min_{\ell}|\alpha_{\ell}|$ and $\max_{\ell}|\beta_{\ell}| \le B$. 
Then, under Assumptions \ref{ass:stab}, for every $\eta > 0$ there exists 
$\Psi_0^{RNN}$ of the form (\ref{eqn:psiRNN}) such that 
$$ \sup_{t \in \mathcal{T}, b,c \in \mathsf{Z}_R}\;\left|\Psi^{PC}_0\bigl(b(t),c(t), \left\{b(\tau)\right\}_{\tau \in \overline{\mathcal{T}}},t; \theta^*\bigr) - \Psi_0^{RNN}\bigl(b(t), c(t), \left\{b(\tau)\right\}_{\tau \in \overline{\mathcal{T}}},t\bigr)\right| < \eta. $$
\end{restatable}

The proof of Theorem \ref{thm:RNN_approx} can be found in Appendix \ref{sapdx:RNN}.


Note that $\Psi_0^{RNN}$ both avoids dependence on the fine-scale $\epsilon$ and is Markovian. The non-homogenized map $\Psi_{\epsilon}^\dagger$ is local in time while the homogenized map $\Psi_0^{RNN}$ is non-local in time and depends on the strain history. Let $u_0^{RNN}$ be the solution to the following system with
constitutive model $\Psi_0^{RNN}$: 
\begin{subequations}\label{eqns:RNN}\begin{align}
   &\rho\p_t^2u_{0}^{RNN}-\p_x\sigma_{0}  = f, \quad  (x,t) \in\Omega  \times \mathcal{T}\label{eqn:RNN-force-balance}\\
    &\sigma_0(t) = \Psi_0^{RNN}\left(\p_xu^{RNN}_{0}(t), \p_{xt}^2  u^{RNN}_{0}(t), \left\{\p_x u^{RNN}_{0}(\tau)\right\}_{\tau \in \overline{\mathcal{T}}},t\right), \quad (x,t) \in \Omega \times \mathcal{T}\\
    &u^{RNN}_{0}\rvert_{t=0}  = u^*, \quad \p_t u^{RNN}_{0}\rvert_{t=0}  = v^*, \quad (x,t) \in \Omega \times \{0\} \\
    &u^{RNN}_{0}  = 0, \quad  (x,t) \in \p\Omega \times \mathcal{T}.
\end{align}
\end{subequations}

Ideally we would like an approximation result bounding $\|u_{\epsilon} - u_0^{RNN}\|_{\mathcal{Z}_2}$, the
difference between solution of the multiscale problem \eqref{eqn:main1d} and the Markovian RNN model
\eqref{eqns:RNN}, in the case $\rho=0.$ Using Corollary \ref{cor:ue_uPC} shows that this would follow from
a bound on $\|u_{0}^{PC} - u_0^{RNN}\|_{\mathcal{Z}}$, where $u_0^{PC}$ solves \eqref{eqns:PCC}, in the
case $\rho=0.$ We note, however, that although
Theorem \ref{thm:RNN_approx} gives us an approximation result between $\Psi_0^{PC}$ and $\Psi_0^{RNN}$, proving that $u_{0}^{PC}$ and $u_{0}^{RNN}$ are close requires developing new theory for the fully nonlinear PDE for $u_0^{RNN}$; developing 
such a theory is beyond the scope of this work. Developing such a theory
is difficult for two primary reasons: (i) the monotonicity property of $\Psi_0^{RNN}$
with respect to strain rate is hard to establish globally, for a trained model;
(ii) the functions $\mathcal{F}^{RNN}, \mathcal{G}^{RNN}$ may not be differentiable. As a result, existence and uniqueness of $u_0^{RNN}$ remains unproven; however, numerical experiments in Section \ref{sec:Numerics} indicate that in practice, $u_0^{RNN}$ does approximate $u_{\epsilon}$ well.

\begin{remark}\hfill
\begin{itemize}
\item Monotonicity of $\Psi_0^{RNN}$ with respect to strain rate is a particular
issue when $\rho=0$ (no inertia) as in this case it is needed to define an
(implicit) equation for $\p_{t} u_0$ to determine the dynamics.
It is for this reason that our experiments
will all be conducted with $\rho>0$, obviating the need for the determination of
an (implicit) equation for $\p_{t} u_0$. However this leads to the issue that
the homogenized equation is only valid for a
subset of initial conditions, in the inertial setting $\rho>0$; see Remark \ref{rem:inertia}.

\item In practice we will train $\Psi_0^{RNN}$ 
using data for the inputs $b,c$ which are
obtained from a random function, or set of realizations of random functions.
The choice of the probability measure from which this training data is drawn
will affect the performance of the learned model when $\Psi_0^{RNN}$ is
evaluated at $b=\p_x u_0^{RNN}$ and $c=\p^2_{xt} u_0^{RNN}$, for displacement $u$ generated by the model given by \eqref{eqns:RNN}.
\end{itemize}
\end{remark}

\subsection{RNN Optimization}\label{ssec:RNN_Optimization}

In the following section, we present numerical results using a trained RNN as a \textit{surrogate model}: an efficient approximation of the complex microscale dynamics; in this section we discuss the problem of finding such an RNN. 
To learn the RNN operator approximation, we are given data 
\begin{equation}
\label{eq:datatt}
\bigl\{\bigl(\p_x u_0\bigr)_n, \bigl(\p^2_{xt} u_0\bigr)_n, \bigl(\sigma_0\bigr)_n\}_{n=1}^N
\end{equation}
where the suffix $n$ denotes the $n^{\text{th}}$ strain, strain rate, and stress trajectories over the entire time interval $\mcT$. Each strain trajectory $(\p_x u_0)_n$ is drawn i.i.d. from a measure $ \mu$ on $\mathcal{C}(\mathcal{T};\bbR)$.
There is no need to generate training data on the same time interval $\TT$ as the macroscale model;
we do so for simplicity.

The data for the homogenized constitutive model is given by
$$\sigma_0(t) = \Psi_0^{\dagger}(\p_x u_0(t),\p_{xt}^2 u_0(t), \left\{\p_x u_{0}(\tau)\right\}_{\tau \in \overline{\mathcal{T}}},t),$$
defined via solution of the cell-problem \eqref{eq:cell}; but it may also 
be obtained as the solution to a forced boundary problem on the microscale, 
as stated in the following lemma.

\begin{restatable}{lemma}{lemmaFBP}\label{lem:data_jus}
   Let $u$ solve the equations 
    \begin{subequations}\label{eqns:OG}
    \begin{align}
        \p_y \sigma (y,t) &= 0, \quad (y,t) \in \Omega \times \TT \\
        \sigma(y,t) &= E(y)\p_y u(y,t) + \nu(y) \p_{ty} u(y,t), \quad (y,t) \in \Omega \times \TT \\
        u(0,t) &= 0 , \quad u(1,t) = b(t), \quad (y,t) \in \p\Omega \times \TT \\
        u(y,0) &= 0, \quad y \in \Omega.
    \end{align}
\end{subequations} Then 
    $$ \{\sigma(t)\}_{t=0}^T = \Psi_0^{\dagger}\left(b(t),\p_tb(t), \{b(t)\}_{t=0}^T,t\right)$$
    where $\Psi_0^{\dagger}$ is the map defined in \eqref{eqn:Lap_inv}. 
\end{restatable} The proof can be found in Appendix \ref{sapdx:FBP} and justifies the application of data resulting from this problem to the homogenized model. 
In the following, we denote by $(\widehat{\sigma}_0)_n$ and $\widehat{\dot{\xi}}_n$ the output of $\mcF^{RNN}$ and $\mcG^{RNN}$ on data point $n$:
\begin{align*}
    (\widehat{\sigma}_0)_n(t) & = \mcF^{RNN}\left((\p_x u_0)_n(t), (\p_{xt}^2 u_0)_n(t),\widehat{\xi}_n(t) \right)\\
    \widehat{\dot{\xi}}_n(t) & = \mcG^{RNN}\left((\p_x u_0)_n(t), \widehat{\xi}_n(t)\right), \quad \widehat{\xi}_n(0) = 0.
\end{align*}
To train the RNN, we use the following relative $L^2$ loss function, which should be viewed as
a function of the parameters defining the neural networks $\mcF^{RNN}, \mcG^{RNN}.$

\textbf{Accessible Loss Function:}
\begin{equation}\label{eqn:loss_accs}
   L_1(\{\sigma_0\}_{n=1}^N, \{\widehat{\sigma}_0\}_{n=1}^N)  = \frac{1}{N}\sum_{n=1}^N \frac{\|(\sigma_0)_n -(\widehat{\sigma}_0)_n\|_{L^2(\mcT;\bbR)}}{\|(\sigma_0)_n\|_{L^2(\mcT;\bbR)}}.
\end{equation}
\begin{remark}
To test robustness of our conclusions, we also employed relative and absolute $L^2$ squared 
loss functions. In doing so we did not observe significant differences in the predictive accuracy of the resulting models.
\end{remark}

In the case of a material that is $2$-piecewise-constant on the microscale, we can explicitly write down the analytic form of the solution, and thus can also know the values of the hidden variable $\{\xi_n\}_{n=1}^N$ and its derivative $\{\dot{\xi}_n\}_{n=1}^N$, for each data trajectory as expressed in equation \eqref{eqn:psiPC}. It is intuitive that
training an RNN on an extended data set which includes the hidden variable should
be easier than using the original data set \eqref{eq:datatt}. In order to deepen our understanding of the training process we will include training in a  
$2$-piecewise-constant which uses this hidden data, motivating
the following loss function. Since, in general, the hidden variable is inaccessible in the data, we refer to the resulting loss as the inaccessible relative loss function:

\textbf{Inaccessible Loss Function:}\label{eqn:loss_inacc}
\begin{equation}
\label{eqn:loss_inacc}
   L_2(\{(\sigma_0)_n\}_{n=1}^N, \{(\widehat{\sigma}_0)_n\}_{n=1}^N, \{\dot{\xi}_n\}_{n=1}^N, \{\widehat{\dot{\xi}}_n\}_{n=1}^N)  = \frac{1}{N}\sum_{n=1}^N\left( \frac{\|(\sigma_0)_n -(\widehat{\sigma}_0)_n\|_{L^2(\mcT;\bbR)}}{\|(\sigma_0)_n\|_{L^2(\mcT;\bbR)}} + \frac{\|\dot{\xi}_n -\widehat{\dot{\xi}}_n\|_{L^2(\mcT;\bbR)}}{\|\dot{\xi}_n\|_{L^2(\mcT;\bbR)}}\right).
\end{equation}

This inaccessible loss function helps identify the RNN whose existence in proved in previous sections. 


\section{Numerical Results}\label{sec:Numerics}

The numerical results make the following contributions, all of which guide the
use of machine-learned constitutive models within complex nonlinear problems beyond the
one-dimensional linear viscoelastic setting considered in this work.
\begin{enumerate}[I.]
    \item {\bf Machine-Learned Constitutive Models.} \label{bul:numpt1} We can find  RNNs that yield low-error simulations when used as a surrogate model in the macroscopic 
    system \eqref{eqns:general-dynamics-homog},\eqref{eq:MARKOV} to approximate
    the multiscale system \eqref{eqns:general-dynamics}, in the one dimensional
    KV setting with inertia. We also discuss how in some material parameter settings, inertial effects lead to higher error in the homogenized approximation.
    \item {\bf Choice of Training Data.} \label{bul:numpt2}We describe our choice of data sampling distribution $\mu$ and show that it exhibits desirable properties.
    \item {\bf Effect of Non-Convex Optimization}\label{bul:numpt3}
    When the inaccessible loss function is used for training, the trained RNN exhibits desirable properties of (approximate) linearity in its arguments in the domain of interest, as is proved for the homogenized constitutive model \eqref{eqns:alphaellbetaell} for piecewise-constant materials. When using the accessible loss function, the trained RNN may perform well as a surrogate model without exhibiting linearity in the equation for evolution of the hidden variables. This is attributable to the
    existence of local minimizers of the loss function and highlights the need for caution in
    training constitutive models.
    \item {\bf Model Choice.} \label{bul:numpt4}The correct choice of architecture for the RNN leads to discretization-robustness in time: a model learned with one choice of time discretization 
    $dt$ performs well when tested on another $dt$; this is not true for poor model choices.
    Discretization-robustness can thus be used as a guide to model choice.
\end{enumerate}

In Subsection \ref{ssec:RNN_macro} we demonstrate that in appropriate settings the solution $u_0^{RNN}$ obtained under the dynamics of a trained RNN approximates the true solution $u_{\eps}$ well when used in the macroscopic setting; furthermore,
this RNN is shown to exhibit linearity in its arguments in the domain of interest. We also discuss the error arising from inertial effects. In Subsection \ref{ssec:RNN_standard} we discuss the performance of an RNN learned using the accessible loss function. In Subsection \ref{ssec:discret_invar} we discuss the discretization-robustness property of the RNN and the choice of data sampling distribution $\mu$.

\FloatBarrier
\subsection{RNN as Surrogate Model}\label{ssec:RNN_macro}
In this subsection we discuss two RNNs: RNN ``A" trained using only the inaccessible loss function in equation \eqref{eqn:loss_inacc} and RNN ``B" trained with the standard loss function in equation \eqref{eqn:loss_accs}, but initialized at parameters obtained via training with the inaccessible loss function. Descriptions of these RNNs, and others we will introduce in subsequent subsections, may be found in Table \ref{tab:RNNs}. 
For details on RNN training, see Appendix \ref{assec:RNN_train_test}.

\begin{table}
\centering
    \begin{tabular}{|c|p{14cm}|}
    \hline
        \textbf{RNN} & \textbf{Description} \\
        \hline
        \textbf{A} & Trained on $2$-piecewise-constant media only with inaccessible loss function \eqref{eqn:loss_inacc} \\
        \hline
        \textbf{B} & Trained on $2$-piecewise-constant media; initialized at solution found with inaccessible loss function then trained  with accessible loss function \eqref{eqn:loss_accs}\\
        \hline
        \textbf{C} & Trained on $2$-piecewise-constant media only with accessible loss function \\
        \hline
        \textbf{D} & Trained on continuous media with accessible loss function\\
        \hline
    \end{tabular}
    \caption{RNN Descriptions}
    \label{tab:RNNs}
\end{table}
In the first surrogate model experiment, we subjected the material to sinusoidal boundary forcing of $b(t) = 0.1\sin(2\pi t)$ starting from $0$ initial displacement and velocity. As a ground-truth comparison, we used a traditional finite element solver with periodic domain of width $0.04$, spatial resolution of $h = 0.005$, and time discretization $dt = 0.1h^2$; we refer to this solution as $u_{\eps}$ and name it FEM. In contrast, the RNN-based macroscale computation  employs spatial resolution of $h_{cell} = 0.04$, with a time discretization of $dt = 0.4h_{cell}^2$; for economy of notation; this solution is denoted $u_0^{RNN}$ and named
RNN. We also compared the results to the displacement obtained using as macroscale constitutive model the analytic solution to the cell problem. To
make comparisons we use the relative error given by
\begin{equation}
    e(u_0,u_{\eps})(t) = \frac{\|u_{\eps}(t) - u_0^{RNN}(t)\|_{L^2(\mathcal{D};\bbR)}}{\|u_{\eps}(t)\|_{L^2(\mathcal{D};\bbR)} + 0.01}.
\end{equation}

The relative error plots for RNNs ``A" and ``B" are shown in Figures \ref{fig:macro_sim_sinuA} and \ref{fig:macro_sim_sinuB}. 
In a second experiment, we subjected the material to integrated Brownian motion forcing starting from null initial conditions. The FEM solver uses the same discretizations as in the sinusoidal forcing experiment, and the RNN spatial discretization was $h_{cell} = 0.05$ with time discretization of $dt = 0.4h_{cell}^2$. The results for RNNs ``A" and ``B" are shown in Figures \ref{fig:macro_sim_brownianA} and \ref{fig:macro_sim_brownianB}.

Both sets of experiments show that the RNN-based macroscopic models accurately 
reproduce the microscale FEM simulation at far lower computational cost. The RNN-based
results have some errors in comparison with the microscale simulation,
but the errors are of the same order of magnitude as the errors arising when the exact homogenized constitutive model is used. The initial error between the analytic solution and the FEM solution is due to inertial effects discussed in Remark \ref{rem:inertia}. The inertial errors become more significant with the ratio between $E$ and $\nu$ varies more across the interval. Relative error results for an RNN trained on a material with more inertial effects is shown in Figure \ref{fig:RNN_inertial}.

Both RNNs also exhibit the desirable property of linearity in the inputs in an appropriate domain as discussed in Section \ref{ssec:RNN_Optimization}; this is presented in Figure \ref{fig:RNN_behavior}.

\begin{figure}[hb!]
        \begin{subfigure}[b]{0.32\textwidth}
            \includegraphics[width = \textwidth]{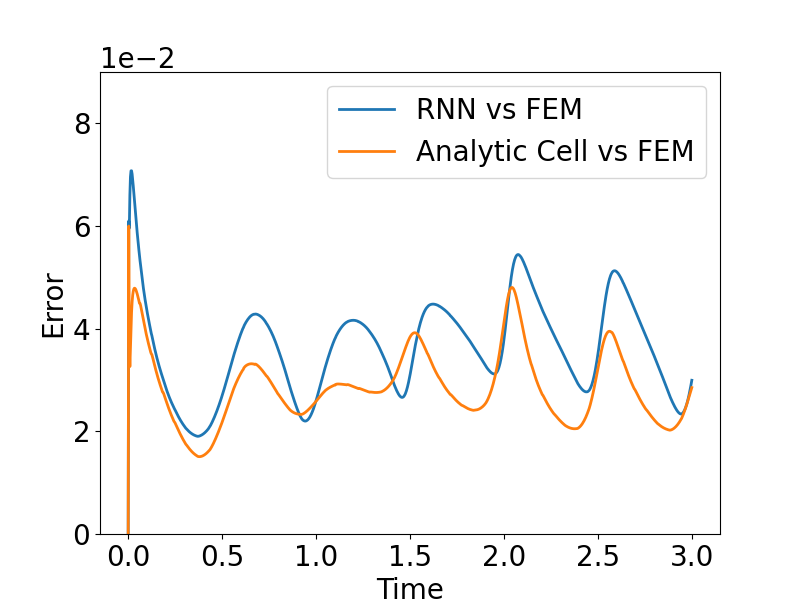}
            \caption{RNN ``A"}
            \label{fig:macro_sim_sinuA}
        \end{subfigure}
        \begin{subfigure}[b]{0.32\textwidth}
            \includegraphics[width = \textwidth]{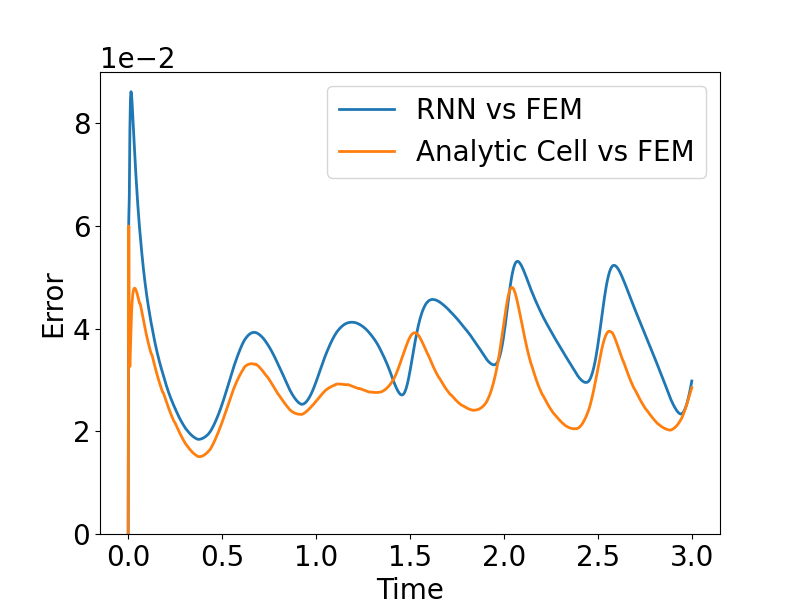}
            \caption{RNN ``B"}
            \label{fig:macro_sim_sinuB}
        \end{subfigure}
        \begin{subfigure}[b]{0.32\textwidth}
            \includegraphics[width = \textwidth]{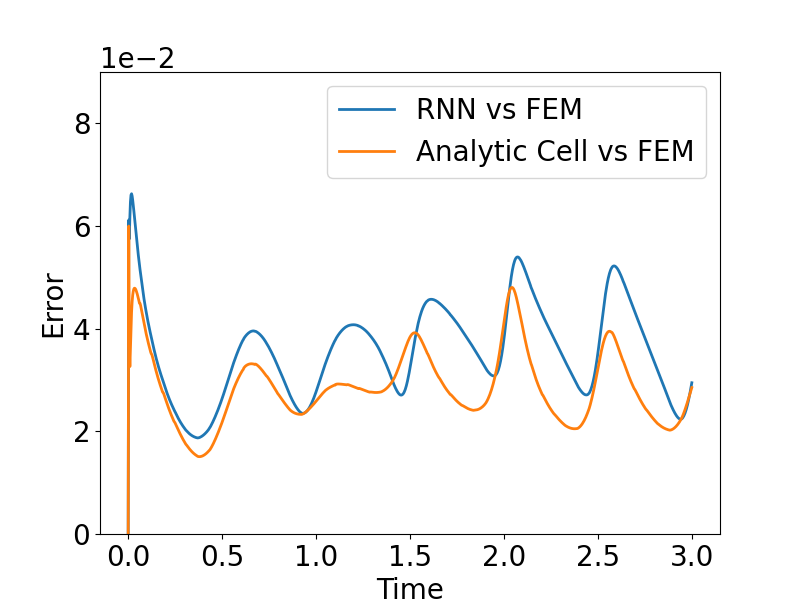}
            \caption{RNN ``C"}
            \label{fig:macro_sim_sinuC}
        \end{subfigure}
        \centering
        \caption{Analytic cell and RNN relative error versus FEM solution using sinusoidal forcing; this supports Numerical Experiments, conclusion \ref{bul:numpt1}.}
        \label{fig:macro_sim_sinu}
\end{figure}

\begin{figure}[ht!]
        \begin{subfigure}[b]{0.32\textwidth}
            \includegraphics[width = \textwidth]{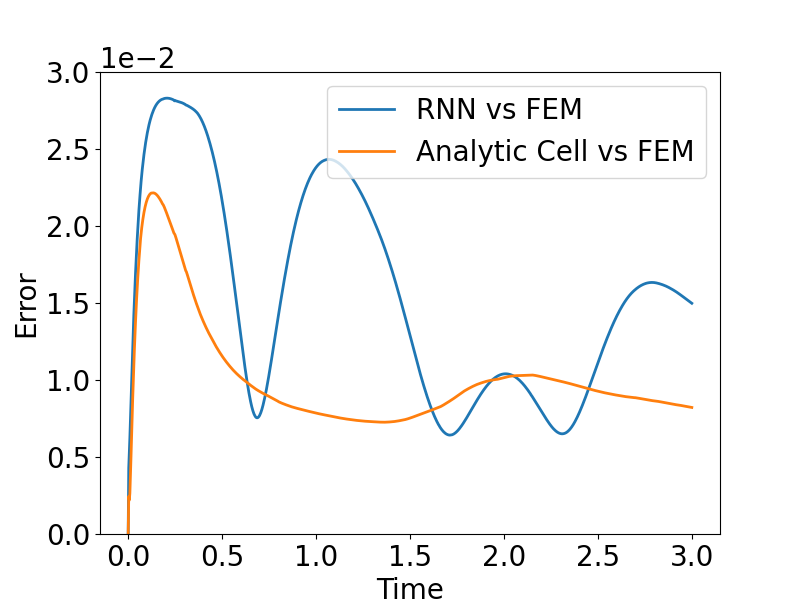}
            \caption{RNN ``A"}
            \label{fig:macro_sim_brownianA}
        \end{subfigure}
        \begin{subfigure}[b]{0.32\textwidth}
            \includegraphics[width = \textwidth]{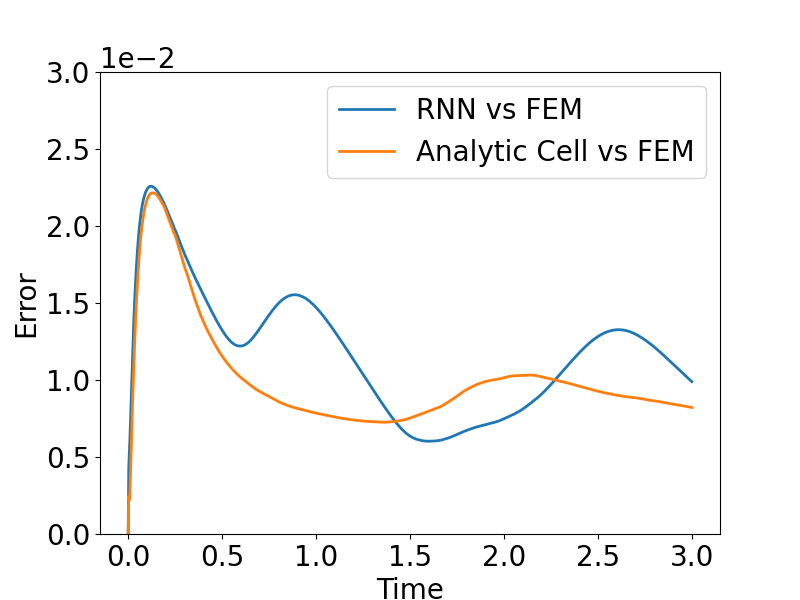}
            \caption{RNN ``B"}
            \label{fig:macro_sim_brownianB}
        \end{subfigure}
        \begin{subfigure}[b]{0.32\textwidth}
            \includegraphics[width = \textwidth]{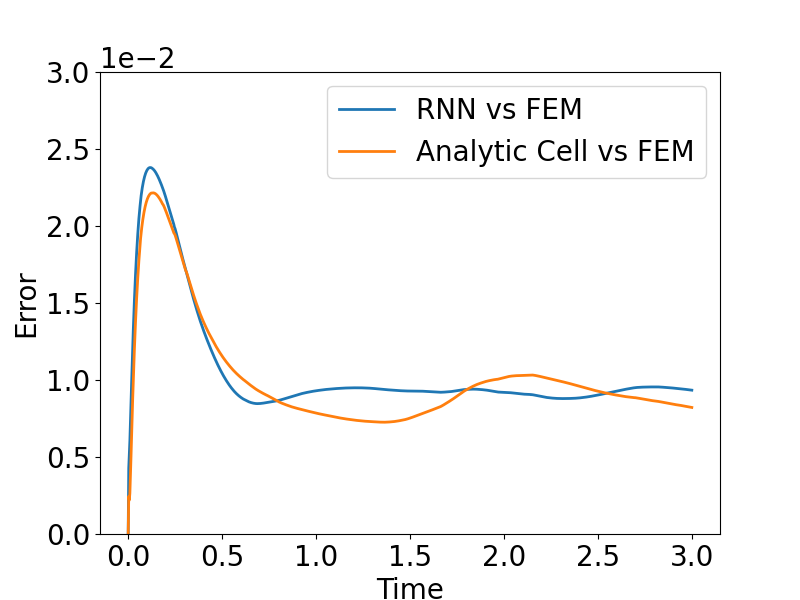}
            \caption{RNN ``C"}
            \label{fig:macro_sim_brownianC}
        \end{subfigure}
        \centering
        \caption{Analytic cell and RNN relative error versus FEM solution using integrated Brownian motion forcing; this supports Numerical Experiments, conclusion \ref{bul:numpt1}.}
        \label{fig:macro_sim_brownian}
\end{figure}

\begin{figure}[ht!]
        \begin{subfigure}[b]{0.49\textwidth}
            \includegraphics[width = \textwidth]{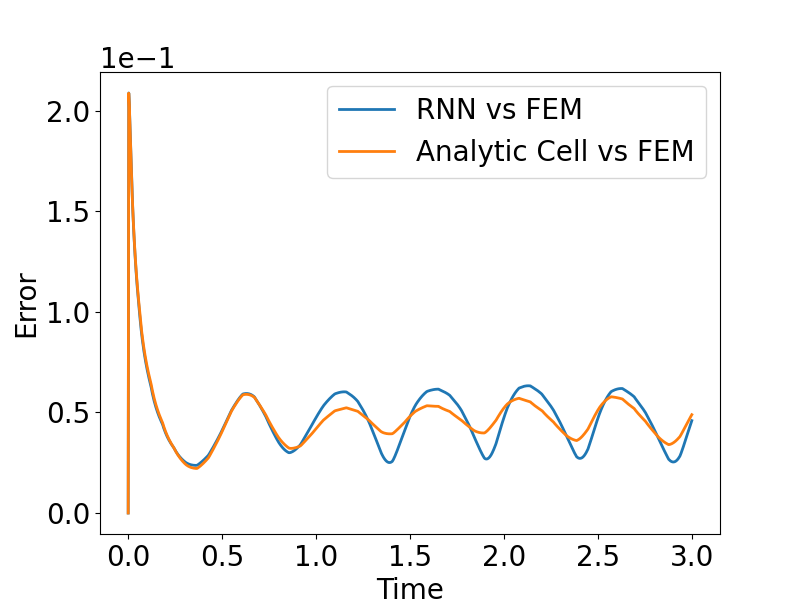}
            \caption{Sinusoidal forcing}
        \end{subfigure}
        \begin{subfigure}[b]{0.49\textwidth}
            \includegraphics[width = \textwidth]{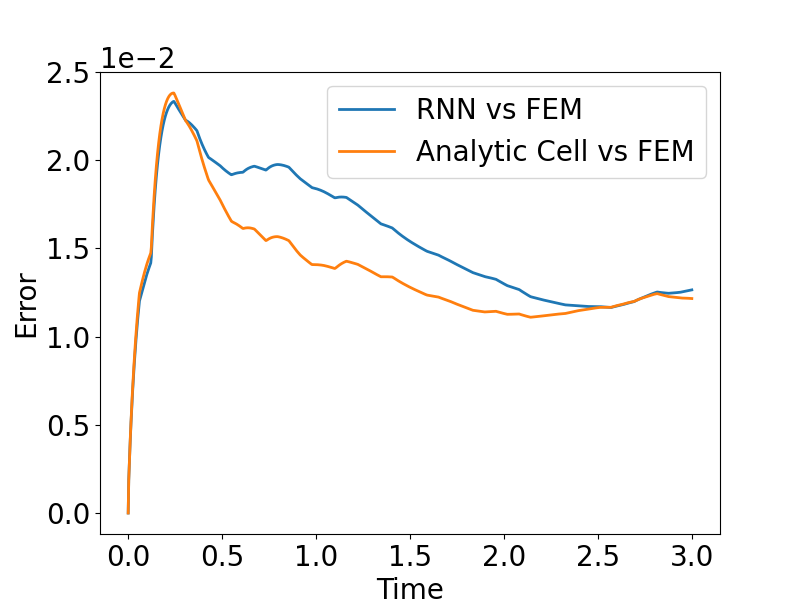}
            \caption{Integrated Brownian motion forcing}
        \end{subfigure}
        \centering
        \caption{Relative error of RNN trained on material parameters with higher inertial effects in response to sinusoidal and integrated Brownian motion forcing; this demonstrates Numerical Experiments, conclusion \ref{bul:numpt1}.}
        \label{fig:RNN_inertial}
\end{figure}

\begin{figure}
    \centering
    \includegraphics[width = \textwidth]{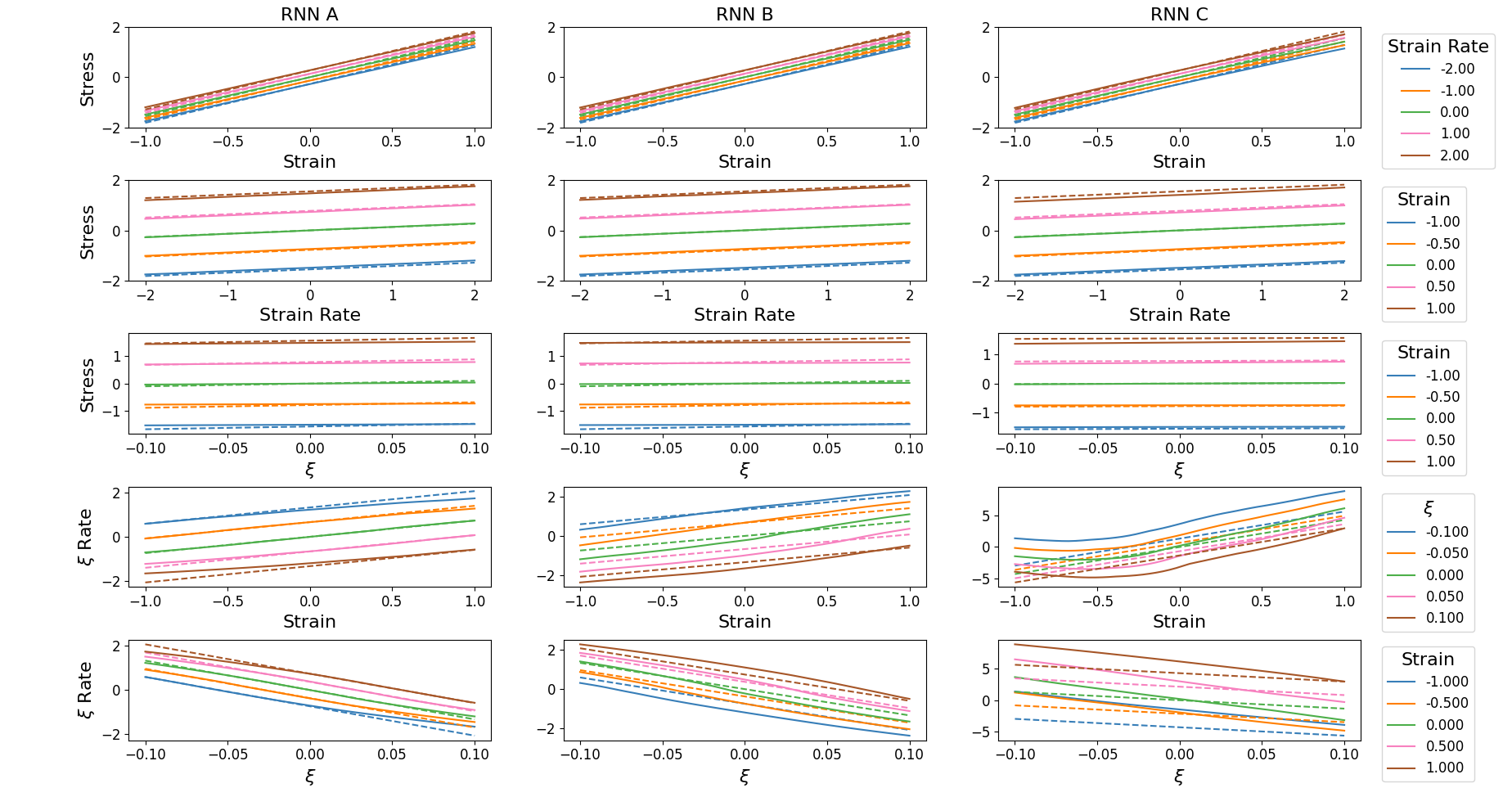}
    \caption{RNN outputs versus the truth (dashed) for each of the three candidate RNNs. The columns correspond to RNNS ``A," ``B," and ``C" respectively. The first row shows the strain-stress dependence for five fixed strain rate inputs. The second row shows the strain rate-stress dependence for five fixed strain inputs. The third row shows the $\xi$, stress relationships for hidden variable $\xi$ for five fixed strain inputs. The fourth row shows the strain, $\dot{\xi}$ relationship for five different fixed values of $\xi$. Finally, the fifth row shows the $\xi,\dot{\xi}$ relationship for five fixed strain inputs. This supports Numerical Experiments, conclusion \ref{bul:numpt3}.}
    \label{fig:RNN_behavior}
\end{figure}

\FloatBarrier

\subsection{RNN Trained with Standard Loss}\label{ssec:RNN_standard}

We also trained a third RNN, denoted RNN ``C," using only the accessible loss function. Details of training may be found in Appendix \ref{assec:RNN_train_test}. The performance of this RNN as a surrogate model in the macroscale simulation experiments may be seen in Figures \ref{fig:macro_sim_sinuC} and \ref{fig:macro_sim_brownianC}. While this RNN performed well as a surrogate model, indeed is comparable in errors
to those of RNNs ``A" and ``B", it does not exhibit a close linear match to the known analytic expression for $\mcG$, as shown in Figure \ref{fig:RNN_behavior}. In the figure, all three RNNs approximate the linear structure of $\mathcal{F}$ well; the difficulty is in obtaining the correct linear dependence in the hidden variable rate, $\dot{\xi}$. Interestingly, by changing the material parameter $\nu_2$ from $0.2$ to $2$, training via the method of RNN ``C" with only the accessible loss function yields an RNN that matches the true linear dependence in $\mathcal{G}$ very well. However, in this parameter regime, inertial effects perturb the simulations on the macroscale to an unacceptable degree,  meaning that the homogenization theory that we use
as benchmark is not valid, and so we avoid this regime. The inability of RNN ``C" to capture the exact linear dependence in $\mathcal{G}$ is unsurprising; indeed, had we guaranteed convergence to the optimal function for any choice of material parameters, we would have entirely sidestepped the problem of high-dimensional optimization inherent to machine learning. 
\begin{figure}[hb!]
        \begin{subfigure}[b]{0.49\textwidth}
            \includegraphics[width = \textwidth]{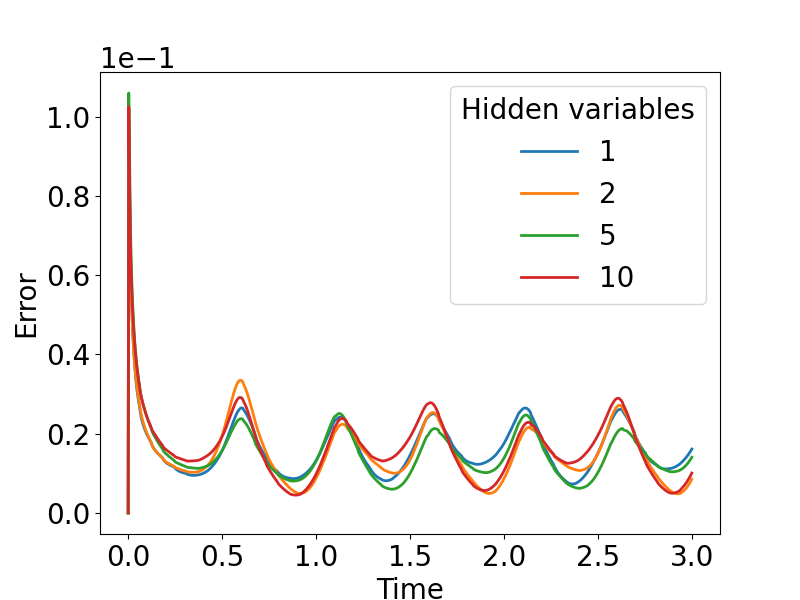}
            \caption{Sinusoidal forcing}
        \end{subfigure}
        \begin{subfigure}[b]{0.49\textwidth}
            \includegraphics[width = \textwidth]{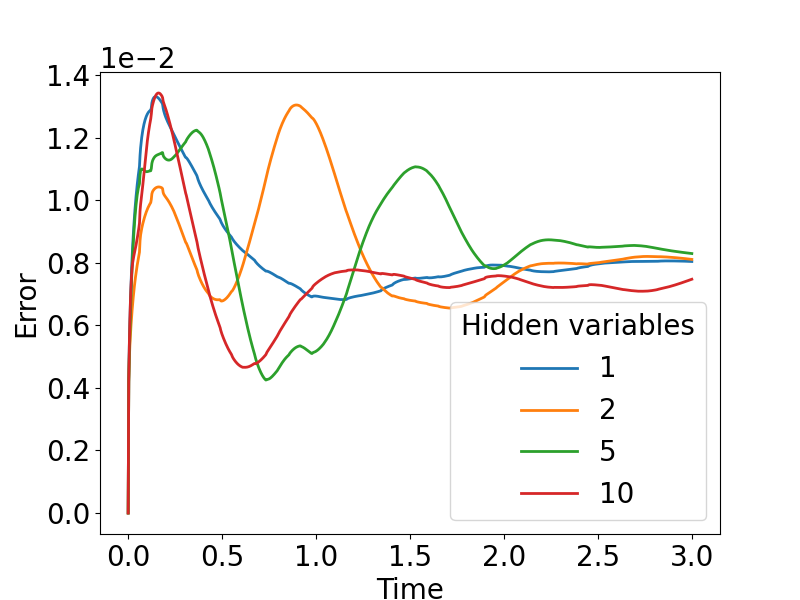}
            \caption{Integrated Brownian motion forcing}
        \end{subfigure}
        \centering
        \caption{Relative error of continuous-material RNNs ``D" with different numbers of hidden variables when used as a surrogate model in a the macroscale system; this supports Numerical Experiments, conclusion \ref{bul:numpt1}.}
        \label{fig:RNND}
\end{figure}
In the case of continuous material properties, we do not have a known analytic solution to the microscale problem and thus do not have access to the hidden variable $\xi$ in the train and test data; in this case, we may only use the accessible loss function. We trained RNNs, denoted RNN type ``D" on continuous media with different numbers of hidden variables $\xi_{\ell}$ and used the trained RNNs as surrogate models in the macroscale system subjected to boundary forcing. Training details may be found in Appendix \ref{assec:RNN_cont}. The relative error of RNN ``D" for the sinusoidal and Brownian motion forcing experiments described previously is shown in Figure \ref{fig:RNND}. We note that similar error is found with all dimensions of the
hidden variable, suggesting that $1$ hidden variable suffices in this case; the
fact that the error does not decrease suggests that the error we see is primarily
from homogenization rather than piecewise-constant approximation.

\FloatBarrier
\subsection{Time Discretization and RNN Training}\label{ssec:discret_invar}
Discretization-robustness is a desirable feature of an RNN surrogate model. To test robustness to time discretization we work in the piecewise-constant media setup leading to RNNs ``A'', ``B'' and ``C''. We evaluated the test error when employing each of the three RNNs using values of timestep $dt$ different from
those used in the training. Additionally, to demonstrate the value of postulating
the correct model form, we trained three additional RNNs via the same methods as described in Subsection \ref{ssec:RNN_standard} but without strain rate dependence.
Figure \ref{fig:time_disc} shows that all three RNNs trained with strain rate as an input parameter were more robust to changes in time discretization than their non-strain-rate counterparts. 

\begin{figure}[ht!]
        \begin{subfigure}[b]{0.32\textwidth}
            \includegraphics[width = \textwidth]{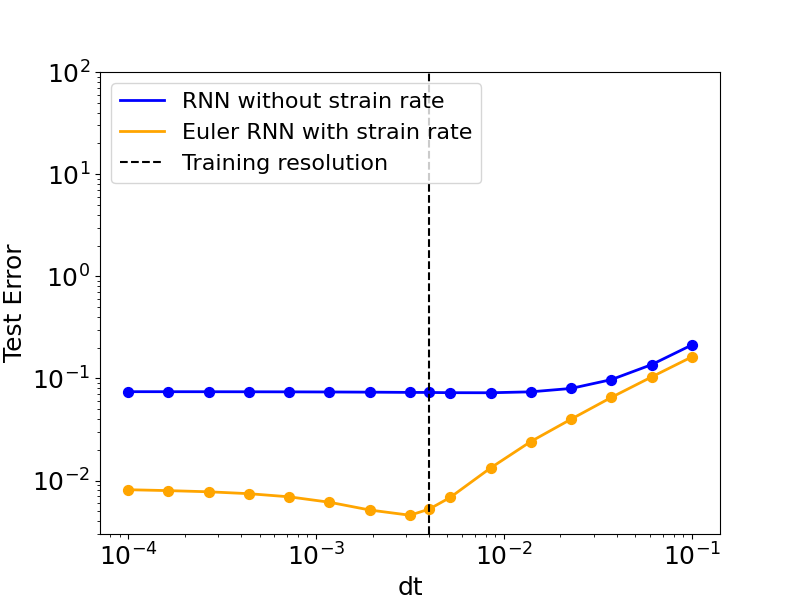}
            \caption{RNN ``A"}
        \end{subfigure}
        \begin{subfigure}[b]{0.32\textwidth}
            \includegraphics[width = \textwidth]{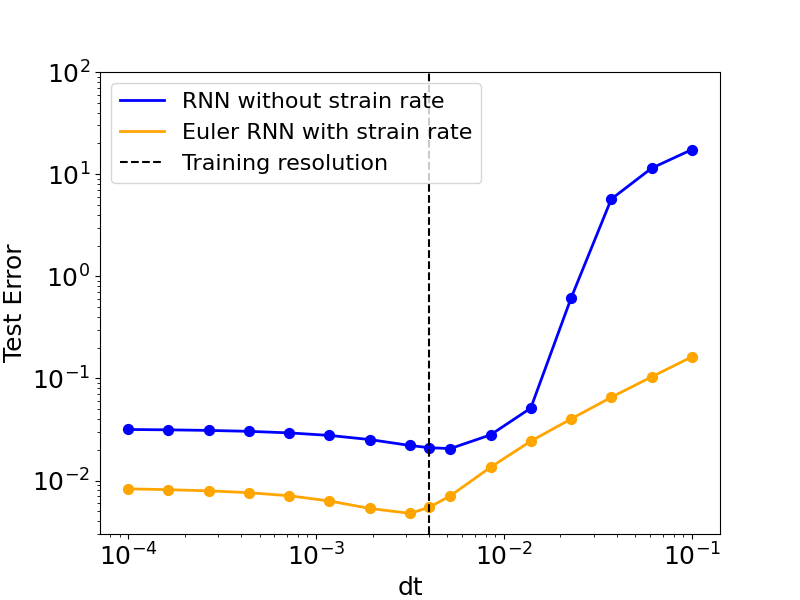}
            \caption{RNN ``B"}
        \end{subfigure}
        \begin{subfigure}[b]{0.32\textwidth}
            \includegraphics[width = \textwidth]{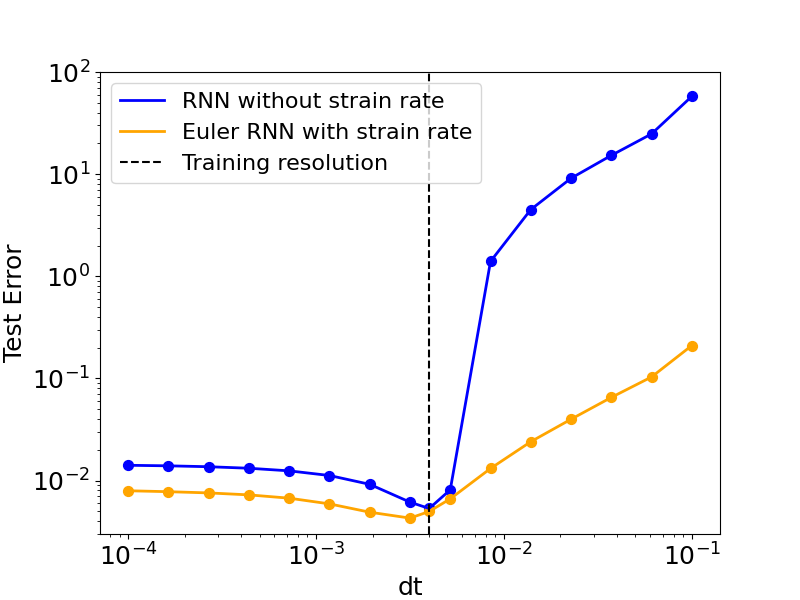}
            \caption{RNN ``C"}
        \end{subfigure}
        \centering
        \caption{Time discretization error for RNNs ``A," ``B," and ``C." This supports Numerical Experiments, conclusion \ref{bul:numpt4}.}
        \label{fig:time_disc}
\end{figure}

\begin{wrapfigure}{R}{0.5\textwidth}
    \centering
    \vspace{-20pt}
    \includegraphics[width = 0.45\textwidth]{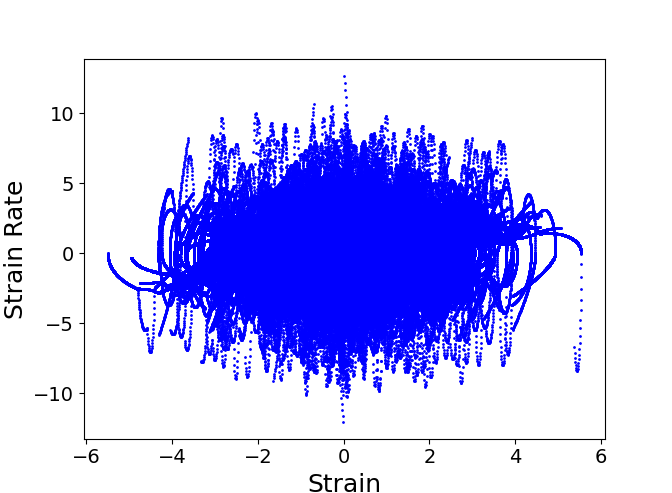}
    \caption{Strain and strain rate distributions in the training data generated by $\mu$; this supports the choice of training data discussed in Numerical Experiments, conclusion \ref{bul:numpt2}.}
    \label{fig:mu}
    \vspace{-20pt}
\end{wrapfigure}

To generate the training strain, we sampled trajectories as follows: first, we randomly partitioned the time interval $\mcT$ into $10$ pieces; second, at each point between these time intervals, we generated a value of strain via a balanced random walk from the previous value scaled by the length of the time interval; third, we used a piecewise cubic hermite interpolating polynomial (pchips) function to interpolate between these values of strain. This choice of distribution has the desirable property that it generates data with a variety of strain/strain-rate pairings evenly dispersed throughout the domain of interest rather than introducing large correlations between the two. A scatterplot of the associated values is shown in Figure \ref{fig:mu}.

\section{Conclusions}
In this paper, we develop theory underlying the learning of Markovian models for history-dependent
constitutive laws. The theory presented applies to the one-dimensional KV case, but the underlying ideas extend to more complex systems. In \cite{biswriting}, numerical experiments suggest that the methodology can be useful in the case of plasticity. Conclusions drawn from our numerical
experiments, underpinned by the theory of this paper and enumerated at the start of Section \ref{sec:Numerics}, provide useful guidance for these more complex nonlinear models in higher spatial dimensions.

Several research directions are suggested by this work. Firstly, when inertial effects are significant, the homogenization theory used in this paper, and underlying the computational work in
\cite{biswriting}, is not valid; extending the methodology to this setting would be useful.
Second, the development of theoretical guidance and methodology for choice of training measure $\mu$
will be very important in this field. Third, as is the case with most machine learning applications, convergence of the RNN to the globally optimal solution is not guaranteed; 
considering learning techniques from reservoir computing could be useful to alleviate this issue
as they lead to convex quadratic optimization problems.

\vspace{0.1in}
\noindent{\bf Acknowledgements} The work of KB, BL and AMS 
was sponsored by the Army Research Laboratory, United States,
and was accomplished under Cooperative Agreement Number W911NF-12-2-0022.
MT is funded by the Department
of Energy Computational Science Graduate Fellowship under Award Number DE-
SC002111.
The authors are grateful to Matt Levine for helpful discussions about
training RNN models and to Pierre Suquet for helpful discussions about 
homogenization theory.
\vspace{0.1in}

\bibliographystyle{unsrt}
\bibliography{refs}

\appendix
\addcontentsline{toc}{section}{Appendices}
\section*{Appendices}

\section{Proofs}\label{apdx:proofs}

\subsection{Proof of Theorem \ref{thm:PC_approx}}\label{sapdx:PC_approx}

The proof of Lemma \ref{lem:u1u2}, which underlies the proof of Theorem \ref{thm:PC_approx},
uses the following two propositions

\begin{restatable}{proposition}{qboundsprop}
\label{prop:bounds_u}
Under Assumptions \ref{ass:stab}, for all solutions $u$ of equation \eqref{eqn:gen_q} the following bounds hold for some constant $C_1$: 
\begin{enumerate}
    \item $\sup_{t \in \mathcal{T}} \|u\|_{H_0^1,\nu}^2 \leq \|u|_{t=0}\|_{H_0^1,\nu}^2+ \left(\frac{\nu^+}{E^-}\right)^2\frac{1}{\nu^-}C_1^2\|f\|^2_{\mathcal{Z}}$
    \item $\sup_{t \in \mathcal{T}} \|u\|_{H_0^1, E}^2 \leq \frac{E_+}{\nu_-}\|u|_{t=0}\|_{H_0^1,\nu}^2+ \left(\frac{\nu^+}{E^-}\right)^2\frac{E^+}{(\nu^-)^2}C_1^2\|f\|_{\mathcal{Z}}^2$
    \item $\|\p_t u\|_{H_0^1, \nu} \leq \frac{C_1\|f\|_{\mathcal{Z}}}{\nu^-} + \frac{E^+}{\nu^-}\|u\|_{H_0^1, E}$, for all ${t\in \mathcal{T}}.$
\end{enumerate}
\end{restatable}
\begin{proof}
To show the first bound, let $\varphi = u$ in equation \eqref{eqn:gen_q}. We have
$$q_{\nu}(\p_t u, u) + q_E(u,u) = \la f, u \ra $$
so that 
\begin{align*}
    \frac{1}{2}\frac{d}{dt}\|u\|_{H_0^1, \nu}^2 + \|u\|_{H_0^1, E}^2 &\leq \|f\|_{H^{-1}}\|u\|_{H_0^1}\\
    &\leq C_1\|f\|\|u\|_{H_0^1}
\end{align*}
for some constant $C_1$, by compact embedding. Then, using Lemma \ref{lem:h01equi},
$$
    \frac{1}{2}\frac{d}{dt}\|u\|_{H_0^1, \nu}^2 + \frac{E^-}{\nu^+}\|u\|_{H_0^1, \nu}^2  \leq \frac{C_1}{2\delta^2}\|f\|^2 + \frac{\delta^2}{2\nu^-}\|u\|_{H_0^1,\nu}^2
$$
for any $\delta >0$ by Young's Inequality. Letting $\delta^2 = \frac{E^- \nu^-}{\nu^+}$, we have
\begin{align*}
    \frac{d}{dt}\|u\|_{H_0^1, \nu}^2 + \frac{E^-}{\nu^+}\|u\|_{H_0^1, \nu}^2 & \leq \frac{C_1^2\nu^+}{E^-\nu^-}\|f\|_{\mathcal{Z}}^2.
\end{align*}
Finally, Gronwall's inequality yields 
\begin{equation*}
    \sup_{t \in \mathcal{T}} \|u\|_{H_0^1, \nu}^2  \leq \|u|_{t=0}\|_{H_0^1,\nu}^2+\left(\frac{\nu^+}{E^-}\right)^2 \frac{C_1^2}{\nu^-}\|f\|^2.
\end{equation*}
The second bound follows from Lemma \ref{lem:h01equi}.
For the third bound, let $\varphi = \p_t u$ in equation \eqref{eqn:gen_q}. Then 
$$q_{\nu}(\p_t u, \p_t u) + q_E(u,\p_t u) = \la f, \p_t u \ra$$
so that, again using Lemma \ref{lem:h01equi}, and using the Poincar\'e inequality,
$$\|\p_t u\|_{H_0^1, \nu}^2  \leq \frac{C_1}{\nu^-}\|f\|\|\p_t u\|_{H_0^1, \nu} + \frac{E^+}{\nu^-}\|u\|_{H_0^1, E}\|\p_t u\|_{H_0^1, \nu}$$
and
$$\|\p_t u\|_{H_0^1, \nu} \leq \frac{C_1}{\nu^-}\|f\|_{\mathcal{Z}} + \frac{E^+}{\nu^-}\|u\|_{H_0^1, E}.$$
\end{proof}
Additionally, we need to bound the difference between two solutions $u_1$ and $u_2$ of the PDE in Lemma \ref{lem:u1u2} with different material properties. 
Notice that $u_1$ and $u_2$ satisfy 
\begin{align*}
    \frac{\p}{\p x}\left(E_1 \left(\frac{\p}{\p x} u_1\right)+\nu_1\left(\frac{\p^2}{\p t\p x} u_1\right)\right)& = -f \\
    \frac{\p}{\p x}\left(E_1 \left(\frac{\p}{\p x} u_2\right)+\nu_1\left(\frac{\p^2}{\p t\p x} u_2\right)\right)& = -f + \frac{\p}{\p x}\left[(E_1 - E_2)\frac{\p}{\p x} u_2 + (\nu_1 - \nu_2) \frac{\p^2}{\p t\p x}u_2\right]
\end{align*}
Subtracting yields 
\begin{align*}
    \p_x\left[E_1 \p_x \gamma + \nu_1 \p^2_{t,x} \gamma\right]& = -\p_x\left[(\Delta E)\p_x u_2 + (\Delta \nu)\p^2_{t,x} u\right]
\end{align*}
where $\gamma = u_1-u_2$, $\Delta E = E_1 - E_2$, and $\Delta \nu = \nu_1 - \nu_2$. 
We can rewrite this as an equation for $\gamma$ in weak form: for all test functions $\varphi \in V$ 
\begin{equation}\label{eqn:gen_q_gamma}
    q_{\nu_1}(\p_t \gamma, \varphi) + q_{E_1}(\gamma, \varphi) = \la g, \p_x \varphi \ra, \quad \gamma|_{t=0}=0,
\end{equation}
where $g = \Delta E \p_x u_2 + \Delta \nu\p^2_{t,x}u_2$. For the following discussion of bounds including both $u_1$ and $u_2$, let $E^+ = \max\{E_1^+, E_2^+\}$, $\nu^+ = \max\{\nu_1^+, \nu_2^+\}$, $E^- = \min \{E_1^-, E_2^-\}$, and $\nu^- = \min \{\nu_1^-, \nu_2^-\}$.


\begin{restatable}{proposition}{gammaboundsprop}
\label{prop:bound_u1u2}
Under Assumptions \ref{ass:stab}, for all solutions $\gamma$ of equation \eqref{eqn:gen_q_gamma}, the following bounds hold: 
\begin{enumerate}
    \item $\sup_{t \in \mathcal{T}} \|\gamma\|_{H_0^1,\nu_1}^2 \leq \left(\frac{\nu^+}{E^-}\right)^2\frac{1}{\nu^-}\|g\|_{\mathcal{Z}}^2$
    \item $\sup_{t \in \mathcal{T}} \|\gamma\|_{H_0^1, E_1}^2 \leq \left(\frac{\nu^+}{E^-}\right)^2\frac{E^+}{(\nu^-)^2}\|g\|_{\mathcal{Z}}^2$
    \item $\sup_{t\in \mathcal{T}} \|\p_t \gamma\|_{H_0^1, \nu_1} \leq \frac{\|g\|_{\mathcal{Z}}}{\nu^-} + \frac{E^+}{\nu^-}\|\gamma\|_{H_0^1, E}$
\end{enumerate}
\end{restatable}

\begin{proof}
To show the first bound, let $\varphi = \gamma$ in equation \eqref{eqn:gen_q_gamma}. We have
$$q_{\nu}(\p_t \gamma, \gamma) + q_E(\gamma,\gamma) = \la g, \p_x \gamma \ra $$
so that 
\begin{align*}
    \frac{1}{2}\frac{d}{dt}\|\gamma\|_{H_0^1, \nu}^2 + \|u\|_{H_0^1, E}^2 &\leq \|g\|\|\gamma\|_{H_0^1}\\
\end{align*}
Then
$$
    \frac{1}{2}\frac{d}{dt}\|\gamma\|_{H_0^1, \nu}^2 + \frac{E^-}{\nu^+}\|\gamma\|_{H_0^1, \nu}^2  \leq \frac{1}{2\delta^2}\|g\|^2 + \frac{\delta^2}{2\nu^-}\|\gamma\|_{H_0^1,\nu}^2
$$
for any $\delta >0$ by Young's Inequality. Letting $\delta^2 = \frac{E^- \nu^-}{\nu^+}$, we have
\begin{align*}
    \frac{d}{dt}\|\gamma\|_{H_0^1, \nu}^2 + \frac{E^-}{\nu^+}\|\gamma\|_{H_0^1, \nu}^2 & \leq \frac{\nu^+}{E^-\nu^-}\|g\|_{\mathcal{Z}}^2.
\end{align*}
Finally, since $\gamma(0) = 0$, Gronwall's inequality yields 
\begin{equation*}
    \sup_{t \in \mathcal{T}} \|\gamma\|_{H_0^1, \nu}^2  \leq \left(\frac{\nu^+}{E^-}\right)^2 \frac{1}{\nu^-}\|g\|^2.
\end{equation*}
The second bound follows from Lemma \ref{lem:h01equi}. 
For the third bound, let $\varphi = \p_t \gamma$ in equation \eqref{eqn:gen_q_gamma}. Then 
$$q_{\nu}(\p_t \gamma, \p_t \gamma) + q_E(\gamma,\p_t \gamma) = \la g, \p_{xt}^2 \gamma \ra$$
so that, again using Lemma \ref{lem:h01equi},
$$\|\p_t \gamma\|_{H_0^1, \nu}^2  \leq \frac{1}{\nu^-}\|g\|\|\p_t \gamma\|_{H_0^1, \nu} + \frac{E^+}{\nu^-}\|\gamma\|_{H_0^1, E}\|\p_t \gamma\|_{H_0^1, \nu}$$
and
$$\|\p_t \gamma\|_{H_0^1, \nu} \leq \frac{1}{\nu^-}\|g\|_{\mathcal{Z}} + \frac{E^+}{\nu^-}\|\gamma\|_{H_0^1, E}.$$
\end{proof}

To prove the Lipschitz property of the solution in Theorem \ref{thm:PC_approx}, we will need the following lemma. 
\lemmau*
\begin{proof}
Let $\gamma$ and $g$ be as defined before and after equation
\ref{eqn:gen_q_gamma}. Then, by the result of 
Proposition \ref{prop:bound_u1u2},
\begin{equation*}
\sup_{t \in \mathcal{T}} \|\gamma\|_{H_0^1}^2 \le
   \frac{1}{\nu^-} \sup_{t \in \mathcal{T}} \|\gamma\|_{H_0^1, \nu_1}^2 \leq \left(\frac{\nu^+}{E^-\nu^-}\right)^2\|g\|^2_{\mathcal{Z}}.
\end{equation*}
To bound the RHS: 
\begin{align*}
    \|g\|_{\mathcal{Z}} &= \|(\Delta E)\p_x u_2 + (\Delta \nu)\p^2_{t,x}u_2\|_{\mathcal{Z}} \\
    & \leq \|(\Delta E)\p_x u_2 \|_{\mathcal{Z}} + \|(\Delta \nu)\p^2_{t,x}u_2\|_{\mathcal{Z}}\\
    & \leq \|\Delta E\|_{\infty}\|\p_x u_2\|_{\mathcal{Z}} + \|\Delta \nu\|_{\infty}\|\p^2_{t,x}u_2\|_{\mathcal{Z}}\\
    & \leq 
    \sup_{t \in \mathcal{T}}
    \|u_2\|_{H^1_0}\|\Delta E\|_{\infty} + 
    \sup_{t \in \mathcal{T}}\|\p_{t}u_2\|_{H^1_0}\|\Delta \nu\|_{\infty}\\
    & \leq 
    \frac{1}{(\nu^-)^{\frac12}}\Bigl(\sup_{t \in \mathcal{T}}
    \|u_2\|_{H^1_0,\nu_2}\|\Delta E\|_{\infty} + 
    \sup_{t \in \mathcal{T}}\|\p_{t}u_2\|_{H^1_0,\nu_2}\|\Delta \nu\|_{\infty}\Bigr)\\
    \end{align*}
To bound $\|\p_xu_2\|_{\mathcal{Z}}$ and $\|\p^2_{t,x}u_2\|$, note that any solution $u_2$ will satisfy equation \eqref{eqn:gen_q} for $(u,E,\nu) \mapsto (u_2,E_2,\nu_2)$. 
The analysis of Proposition \ref{prop:bounds_u} yields 
\begin{align*}
     \sup_{t\in \mathcal{T}} \|u_2\|_{H_0^1, \nu_2} &\leq \|u|_{t=0}\|_{H_0^1,\nu}+ \left(\frac{\nu^+}{E^-}\right)\frac{C_1}{(\nu^-)^{1/2}}\|f\|_{\mathcal{Z}}\\
\end{align*}
and
\begin{align*}
    \sup_{t\in \mathcal{T}} \|\p_t u_2\|_{H_0^1, \nu_2} &\leq \frac{C_1\|f\|_{\mathcal{Z}}}{\nu^-} + \frac{E^+}{\nu^-}\|u_2\|_{H_0^1, E}\\
    & \leq \frac{C_1}{\nu^-}\|f\|_{\mathcal{Z}} + \left(\frac{E^+}{\nu^-}\right)^{3/2}\|u|_{t=0}\|_{H_0^1,\nu} + \frac{(E^+)^{3/2}}{(\nu^-)^2}\frac{\nu^+}{E^-}C_1\|f\|_{\mathcal{Z}}.
\end{align*}
By the Poincar\'e inequality, $\|\gamma\|_{\mathcal{Z}} \leq C_p \sup_{t \in \mcT} \|\gamma\|_{H_0^1}$ for some constant $C_p$ and setting \begin{align*}C = C_p\left(\frac{\nu^+}{E^-(\nu^-)^{\frac32}}\right)\max\Bigr\{&\|u|_{t=0}\|_{H_0^1,\nu}+ \left(\frac{\nu^+}{E^-}\right)\frac{C_1}{(\nu^-)^{1/2}}\|f\|_{\mathcal{Z}},\\ \;&\frac{C_1}{\nu^-}\|f\|_{\mathcal{Z}} + \left(\frac{E^+}{\nu^-}\right)^{3/2}\|u|_{t=0}\|_{H_0^1,\nu} + \frac{(E^+)^{3/2}}{(\nu^-)^2}\frac{\nu^+}{E^-}C_1\|f\|_{\mathcal{Z}}\Bigr\}\end{align*} gives the result.

\end{proof}

Now we can prove the piecewise constant approximation theorem. 
\PCapprox*

\begin{proof} 

Let $\mathcal{A}_{E}$ and $\mathcal{A}_{\nu}$ be the finite sets of discontinuities of $E_{\eps}$ and $\nu_{\eps}$ respectively, and let $\mathcal{A} = \mathcal{A}_E \cup \mathcal{A}_{\nu}$ with elements $a_1, a_2, \dots, a_K$. Partition the interval $\mathcal{D}$ into intervals $D_1 = (a_0,a_1), D_2 = [a_1,a_2),\dots D_K = [a_{K-1},a_K)$ such that $\bigcup_{k=1}^K D_k = \mathcal{D}$ and $\bigcap_{k=1}^K D_k  = 0$. Let $B_{k,\delta} = \{b^k_0, b^k_1, \dots, b^k_{N(\delta)}\}$ be a uniform partition of $D_k$ such that $b^k_i - b^k_{i-1} = \delta$. Furthermore, define $E_{\eps}^{PC}$ and $\nu_{\eps}^{PC}$ via 
\begin{align*}
    E_{\eps}^{PC}(x) &= \sum_{k =1}^K \sum_{n = 1}^N \mathds{1}_{x \in (b^k_{n-1},b^k_n]}E\left(\frac{1}{2}b^k_{n-1}+\frac{1}{2}b^k_n\right) \\
     \nu_{\eps}^{PC}(x) &= \sum_{k =1}^K \sum_{n = 1}^N \mathds{1}_{x \in (b^k_{n-1},b^k_n]}\nu\left(\frac{1}{2}b^k_{n-1}+\frac{1}{2}b^k_n\right)
\end{align*}
for $x \in \mathcal{D}$, noting that $E_{\eps}^{PC}$ and $\nu_{\eps}^{PC}$ are piecewise constant with $KN(\delta)$ pieces. 

$E_{\eps}$ and $\nu_{\eps}$ are continuous on each interval $D_k$, so for each $\delta' > 0$, there exists a mesh width $\delta$ such that with partitions $\{B_{k,\delta}\}_{k=1}^K$
\begin{align*}
    \sup_{x \in (b_{n-1}^k, b_n^k]} \|E_{\eps}\left(\frac{1}{2}b^k_{n-1}+\frac{1}{2}b^k_n\right) - E_{\eps}(x)\| &< \delta'\\
    \sup_{x \in (b_{n-1}^k, b_n^k]} \|\nu_{\eps}\left(\frac{1}{2}b^k_{n-1}+\frac{1}{2}b^k_n\right) - \nu_{\eps}(x)\| &< \delta'
\end{align*}
for all $n \in \{1,\dots, N(\delta)\}$. Thus, $\|E^{PC} -E\|_{\infty} < \delta '$ and $\|\nu^{PC} - \nu\|_{\infty} < \delta'$. Since $\delta ' $ was arbitrary, we can pick $\delta ' < \frac{\eta}{C_1}$ where $C_1$ is as in Lemma \ref{lem:u1u2}, and the theorem follows by use of the same lemma. 
\end{proof}

\subsection{Proof of Theorem \ref{thm:PC_exact}}\label{sapdx:PC_exact}
We will need the following lemma:
\begin{restatable}[Existence of Exact Parametrization]{lemma}{exactpieceparam}\label{lem:exactpieceparam}
For a piecewise constant material with $L' + 1$ pieces and under Assumptions \ref{ass:stab}, $a_0$ in equation \eqref{eqn:a_0} can be written exactly as 
$$ \widehat{a}_0(s) = E' + \nu' s - \sum_{\ell =1}^{L'} \frac{\beta_{\ell}}{s+ \alpha_{\ell}} $$
where $E', \nu', \beta_{\ell} \in \bbR$ and $\alpha_{\ell} \in \bbR_{+}$ for all $\ell \in [L']$. 
\end{restatable}

\begin{proof}
Let $E(y)$ and $\nu(y)$ have $L'+1$ constant pieces of lengths $\{d_{\ell}\}_{\ell=1}^{L'+1}$, each associated to values $\{E_{\ell}\}_{\ell=1}^{L'+1}$ and $\{\nu_{\ell}\}_{\ell=1}^{L'+1}$ of $E$ and $\nu$. Then equation \eqref{eqn:a_0}, rewritten here for convenience
\begin{align*}
    \widehat{a}_0(s) & = \left(\int_0^1 \frac{dy}{s\nu(y) + E(y)}\right)^{-1},
\end{align*}
becomes
\begin{align}
    \widehat{a}_0(s) &= \left[ \sum_{\ell = 1}^{L'+1} \frac{d_{\ell}}{E_{\ell}+\nu_{\ell}s} \right]^{-1}\\
    & = \frac{\prod_{\ell=1}^{L'+1} (E_{\ell}+ \nu_{\ell}s)}{\sum_{\ell = 1}^{L'+1} d_{\ell}\;\prod_{j \neq \ell}(E_j + \nu_j s)}\\
    & =  \frac{P(s)}{Q(s)}
\end{align}
where $P(s)$ is a polynomial of degree $L'+1$ and $Q(s)$ a polynomial of degree $L'$. Therefore, there exists a decomposition 
\begin{equation}\label{eqn:decomp}
    \frac{P(s)}{Q(s)} = E' + \nu's - \frac{C(s)}{Q(s)}
\end{equation}
for some constants $E'$ and $\nu'$ and polynomial $C(s)$ of degree $L'-1$. 

Let $-\alpha_1,\dots,-\alpha_{L'}$ be the roots of $Q(s)$. Then $\frac{C(s)}{Q(s)} = \sum_{\ell=1}^{L'} \frac{\beta_{\ell}}{s+\alpha_{\ell}}$ for some constants $\beta_{\ell}\in \bbC$ by partial fraction decomposition. We wish to show that $\Re(\alpha_{\ell}) > 0$ for all roots $-\alpha_{\ell}$ of $Q(s)$ so that we can take the inverse Laplace transform. Furthermore, we wish to show that, in fact, $-\alpha_{\ell} \in \bbR$ for all roots $-\alpha_{\ell}$ so that $\beta_{\ell} \in \bbR$ as well. Since $E_j$ and $v_j$ are positive for all $j \in [L'+1]$, it is clear that if a root $-\alpha_{\ell}$ is real, then it cannot be positive since $Q(s) = \sum_{\ell=1}^{L'+1} d_{\ell}\prod_{j\neq \ell}(E_j + \nu_js)$ has all positive coefficients. We now show that all roots of $Q(s)$ are real. Suppose $a+bi$ is a root of $Q(s)$. Then
\begin{align*}
    Q(a+bi) & = \sum_{\ell=1}^{L'+1}d_{\ell}\prod_{j \neq \ell} (E_j + \nu_j(a+bi)) \\
    & = \left[\prod_{j = 1}^{L' + 1}(E_j + \nu_j(a+bi))\right]\cdot \sum_{\ell = 1}^{L'+1}\frac{d_{\ell}}{E_{\ell} + \nu_{\ell}(a+bi)} \\
    & = \left[\prod_{j = 1}^{L' + 1}(E_j + \nu_j(a+bi))\right]\cdot \sum_{\ell = 1}^{L'+1}\left(\frac{d_{\ell}(E_{\ell} + \nu_{\ell}a)}{(E_{\ell}+\nu_{\ell} a)^2 + (\nu_{\ell} b)^2} - \frac{d_{\ell}(\nu_{\ell}b)}{(E_{\ell} + \nu_{\ell} a)^2 + (\nu_{\ell} b)^2}i\right). 
\end{align*}
The term $ \prod_{j = 1}^{L' + 1}(E_j + \nu_j(a+bi))$ is a nonzero constant for $b \neq 0$ since $E_j, \nu_j \in \bbR_+$. Therefore, for $Q(a+bi)=0$, both the real and imaginary components of the sum on the right must total $0$. However, since $d_{\ell}$, $\nu_{\ell}$, and the denominator term $(E_{\ell} + \nu_{\ell} a)^2 + (\nu_{\ell} b)^2$ are all positive as well, $b$ must equal $0$ to make $\text{Im}[Q(a+bi)] = 0$. Therefore, all roots of $Q(s)$ are in $ \bbR_-$. 
Returning to the decomposition, we now have 
\begin{equation}
    \widehat{a}_0(s) = E' + \nu's - \sum_{\ell=1}^{L'}\frac{\beta_{\ell}}{s+\alpha_{\ell}}
\end{equation}
where $\beta_{\ell} \in \bbR$ and $\alpha_{\ell} \in \bbR_+$ for all $\ell \in [L']$.
\end{proof}
Now we may prove the theorem.
\exactpieceparamthm*
\begin{proof}
By Lemma \ref{lem:exactpieceparam}, we have that 
\begin{align*}
    \widehat{\sigma_0} & = \widehat{a_0}(s)\p_x \widehat{u_0} \\
    & = \left(E' + \nu' s -  \sum_{\ell =1}^{L'} \frac{\beta_{\ell}}{s+ \alpha_{\ell}}\right) \p_x \widehat{u_0}
\end{align*}
where $\beta_{\ell} \in \bbR$ and $\alpha_{\ell} \in \bbR_+$ for all $\ell \in [L']$. Taking an inverse Laplace transform, we get
\begin{equation}
    \sigma_0(t) = E'\p_x u_0 (t) + \nu' \p_t\p_x u_0(t) - \sum_{\ell = 1}^{L'}\beta_{\ell}\int_0^t \p_x u_0(\tau) \exp[-\alpha_{\ell}(t-\tau)]\;d\tau.
\end{equation}
The above may be reexpressed as equations \eqref{eqn:PC_exact} with a choice of parameters $\theta = (E',\nu',L',\alpha,\beta)$ and auxiliary variable $\xi$.
\end{proof}
\subsection{RNN Approximation Theorem \ref{thm:RNN_approx} Proof}\label{sapdx:RNN}

In this subsection we use $|\cdot|, \|\cdot\|$ to denote modulus and Euclidean norm, respectively,
and $\langle \cdot, \cdot \rangle$ to denote Euclidean inner-product. This overlap with the notation
from Subsection \ref{ssec:N} should not lead to any confusion as it is confined to this subsection.

To prove Theorem \ref{thm:RNN_approx} we first study the simple case where $\mcF^{PC}, \mcG^{PC}$ 
are uniformly approximated across all inputs; subsequently we will use this to establish Theorem \ref{thm:RNN_approx} as stated.

\begin{assumptions}\label{ass:RNN}
For any $\delta > 0$, there exist $\mcF^{RNN}$ and $\mcG^{RNN}$ such that 
\begin{align*}
    &\sup_{z \in \bbR^{2+L'}}\left|\mcF^{PC}(z) - \mcF^{RNN}(z)\right|\leq \delta\\
    &\sup_{z \in \bbR^{1+L'}}\left\|\mcG^{PC}(z) - \mcG^{RNN}(z)\right\|\leq \delta
\end{align*}
\end{assumptions}

\begin{restatable}{proposition}{RNNeasyprop}
\label{prop:RNN_easy}
Under Assumptions \ref{ass:RNN}, if $\{\alpha_{\ell}\}$ in equations \eqref{eqns:alphaellbetaell} are bounded such that $0< \rho < \alpha_{\ell}$ for some $\rho$ for all $\ell\in [L']$, then  for any $\eta > 0$, there exists a map $\Psi_0^{RNN}$ defined as in equations \eqref{eqn:psiRNN} such that for $\Psi_0^{PC}$ defined in equations \eqref{eqn:psiPC}, for any $t \in \bbR^+$ and functions $b,c : \bbR^+\to\bbR$, $$\left|\Psi_0^{PC}(b(t),c(t),\{b(\tau)\}_{\tau \in \overline{\mathcal{T}}},t;\theta^*)-\Psi_0^{RNN}(b(t),c(t),\{ b(\tau)\}_{\tau \in \overline{\mathcal{T}}},t)\right| \leq \eta.$$
\end{restatable}
\begin{proof}
Note that the main difficulty in this proof results from the fact that $\mcF^{RNN}$ and $\mcF^{PC}$ act on different hidden variables $\xi$, which we will denote $\xi^{RNN}$ and $\xi^{PC}$, and whose first order time derivatives are given by $\mcG^{RNN}$ and $\mcG^{PC}$ respectively. We write 
\begin{align*}
   &\left|\Psi_0^{PC}(b(t),c(t),\{b(\tau)\}_{\tau \in \overline{\mathcal{T}}},t;\theta^*)-\Psi_0^{RNN}(b(t),c(t),\{ b(\tau)\}_{\tau \in \overline{\mathcal{T}}},t)\right| \\ & = |\mcF^{PC}(b(t),c(t),\xi^{PC}(t))-\mcF^{RNN}(b(t), c(t), \xi^{RNN}(t))|\\
    & \leq |\mcF^{PC}(b(t), c(t), \xi^{RNN}(t)) - \mcF^{RNN}(b(t), c(t), \xi^{RNN}(t))|\\&+|\mcF^{PC}(b(t), c(t), \xi^{PC}(t)) - \mcF^{PC}(b(t), c(t), \xi^{RNN}(t))| \\
    & \leq \delta + |\mcF^{PC}(b(t), c(t), \xi^{PC}(t)) - \mcF^{PC}(b(t), c(t), \xi^{RNN}(t))|
\end{align*}
by Assumptions \ref{ass:RNN} since $\mcF^{PC}$ and $\mcF^{RNN}$ share arguments in the first term. To bound the second term,
\begin{align*}
     |\mcF^{PC}(b(t), c(t), \xi^{PC}(t)) - \mcF^{PC}(b(t), c(t), \xi^{RNN}(t))| & = |\la \mathds{1},\xi^{PC}(t)-\xi^{RNN}(t) \ra | \leq \sqrt{L'}\|\xi^{PC}(t)-\xi^{RNN}(t)\|
\end{align*}
using the known form of $\mcF^{PC}$ where $\|\cdot \|$ is the Euclidean norm in $\bbR^{L'}$.

Let $e_{\xi}(t) = \xi^{PC}(t) - \xi^{RNN}(t)$. Note that $\xi^{PC}(0) = \xi^{RNN}(0) = 0$, so $e_{\xi}(0) = 0$. We wish to bound $\|e_{\xi}(t)\|$. To do so, we first bound $\|\dot{e}_{\xi}(t)\|$, where $\dot{e}_{\xi}(t) = \frac{d}{dt} e_{\xi}(t)$:
\begin{align*}
    \dot{e}_{\xi}(t) & = \dot{\xi}^{PC}(t) - \dot{\xi}^{RNN}(t)\\
    & = \mcG^{PC}(\xi^{PC}(t),b(t))-\mcG^{PC}(\xi^{RNN}(t),b(t)) - \mcG^{RNN}(\xi^{RNN}(t),b(t)) + \mcG^{PC}(\xi^{RNN}(t),b(t))\\
    & = \mcG^{PC}(\xi^{PC}(t),b(t)) - \mcG^{PC}(\xi^{RNN}(t),b(t)) + q(t)
\end{align*}
where we have defined $q(t) = \mcG^{PC}(\xi^{RNN}(t),b(t))-\mcG^{RNN}(\xi^{RNN}(t),b(t))$ and  $\|q(t)\| \leq \delta$ by Assumptions \ref{ass:RNN}. Now note that $\dot{e}_{\xi}(t) = -Ae_{\xi}(t) + q(t)$ by the form of $\mcG^{PC}$, so we can bound
\begin{align*}
    \frac{1}{2}\frac{d}{dt} \|e_{\xi}(t)\|^2 & = \la e_{\xi}(t), \dot{e}_{\xi}(t) \ra = -\la e_{\xi}(t), Ae_{\xi}(t) \ra + \la q(t), e_{\xi}(t)\ra \\
    & \leq -\alpha_{\min}\|e_{\xi}(t)\|^2 + \left\la \frac{1}{\alpha_{\min}^{1/2}}q(t), \alpha_{\min}^{1/2}e_{\xi}(t)\right\ra\\
    &\leq - \alpha_{\min}\|e_{\xi}(t)\|^2 + \frac{1}{2\alpha_{\min}}\|q(t)\|^2 + \frac{\alpha_{\min}}{2}\|e_{\xi}(t)\|^2\\
    \frac{d}{dt}\|e_{\xi}(t)\|^2 & \leq - \alpha_{\min}\|e_{\xi}(t)\|^2 + \frac{\delta^2}{\alpha_{\min}}
\end{align*}
by Young's inequality. Then by Gronwall's inequality
\begin{align}\label{eqn:thm4_result}
    \|e_{\xi}(t)\|^2 & \leq \frac{\delta^2}{\alpha_{\min}^2}\left(1-e^{-\alpha_{\min}t}\right)
\end{align}
so $\|e_{\xi}(t)\|< \frac{\delta}{\rho}$ for all time. Returning to the main proof narrative, 
\begin{align*}
    &\left|\Psi_0^{PC}(b(t),c(t),\{b(\tau)\}_{\tau \in \overline{\mathcal{T}}},t;\theta^*)-\Psi_0^{RNN}(b(t),c(t),\{ b(\tau)\}_{\tau \in \overline{\mathcal{T}}},t)\right|\\ & \leq \delta + \sqrt{L'}\|\xi^{PC}(t)-\xi^{RNN}(t)\| \leq \delta + \frac{\sqrt{L'}\delta}{\rho}. 
\end{align*}
Since by Assumptions \ref{ass:RNN}, $\delta$ is arbitrarily small, the theorem result is shown.
\end{proof}
Although we did not need to restrict the inputs $t,b$, and $c$ in Proposition \ref{prop:RNN_easy} to compact sets in order to prove it, we will argue that the statement holds under weaker assumptions if the inputs are also bounded. 
The following weaker assumptions follow from the Universal Approximation Theorem for RNNs\cite{cybenko1989approximation}. 
\begin{assumptions}\label{ass:UAT}
If $D_1 \in \bbR^{2+L'}$ and $D_2 \in \bbR^{1+L'}$ are compact sets, then for any $\delta > 0$, there exist $\mcF^{RNN}$ and $\mcG^{RNN}$ such that 
\begin{align*}
    &\sup_{z \in D_1}\left|\mcF^{PC}(z) - \mcF^{RNN}(z)\right|\leq \delta\\
    &\sup_{z \in D_2}\left\|\mcG^{PC}(z) - \mcG^{RNN}(z)\right\|\leq \delta.
\end{align*}
\end{assumptions}

\RNNapprox*

\begin{proof}
Notice first that Assumptions \ref{ass:UAT} are a weaker version of Assumptions \ref{ass:RNN}. We will prove the theorem by showing that, for inputs bounded via $t \in \mcT$ and $b,c \in \mathsf{Z}_R$, we never need the stronger assumption in the proof of Proposition \ref{prop:RNN_easy} because the function arguments of $\mcF^{PC}, \mcF^{RNN},\mcG^{PC},$ and $\mcG^{RNN}$ never leave a compact set. First we show that $\sup_{t \in \mcT} \|\xi^{PC}(t)\| \leq R_3$ for some $R_3 > 0$.  For any $\ell\in \{1,\dots, L'\}$, we have 
\begin{align*}
    \dot{\xi}^{PC}_{\ell}(t) &= \beta_{\ell}b(t) - \alpha_{\ell} \xi_{\ell}^{PC}(t)\\
    |\xi^{PC}_{\ell}(t)| &\leq e^{-\alpha_{\ell}t}\beta_{\ell}\left(\sup_{t \in \mcT}|b(t)|\right)\int_0^t e^{\alpha_{\ell}t'}dt' \\
    & \leq e^{-\alpha_{\ell}t}\beta_{\ell}\left(\sup_{t \in \mcT}|b(t)|\right) \frac{1}{\alpha_{\ell}}e^{\alpha_{\ell}t}\\
    \sup_{t\in \mcT}|\xi^{PC}_{\ell}(t)| &\leq \frac{B}{\rho}R,
\end{align*}
so that $\sup_{t \in \mcT} \|\xi^{PC}(t)\| \leq \frac{\sqrt{L'}BR}{\rho}.$ Let $R_3 = \frac{\sqrt{L'}BR}{\rho}$. Define $R_4 = \max\{R_3 + \frac{\delta}{\rho},R\}$ for $\delta$ in Assumptions \ref{ass:UAT}. We will show that $\sup_{t \in \mcT} \|\xi^{RNN}(t)\|\leq R_4$ for $\xi^{RNN}$ defined by $\xi^{RNN}$ in equations \eqref{eqn:psiRNN}. Then the proof of Proposition \ref{prop:RNN_easy} will apply for bounded $t,b$ and $c$ with the weaker assumptions since all inputs to $\mcF^{PC},\mcF^{RNN},\mcG^{PC}$, and $\mcG^{RNN}$: $b(t)$, $c(t)$, $\xi^{PC}(t)$, and $\xi^{RNN}(t)$, will remain in a compact set for $t \in \mcT$. 

Suppose for the sake of contradiction that there exists a time $t'\in \mcT$ at which $\|\xi^{RNN}(t')\| > R_4$. Then there exists a time $T'< t' <T$ and $\eps > 0$ such that for $t \in [0,T']$, $\|\xi^{RNN}(t)\|\leq R_4$, for $t \in (T',T'+\eps)$, $\|\xi^{RNN}(t)\|>R_4$, and $\|\xi^{RNN}(T')\| = R_4$ by continuity. In other words, $T'$ is the time at which $\xi^{RNN}$ first crosses the $R_4$ radius. Then 
$$\|e_{\xi}(T')\|: = \|\xi^{RNN}(T') - \xi^{PC}(T')\| \geq R_4 - R_3 \geq \frac{\delta}{\rho}$$ 
by triangle inequality.  Since $\|\xi^{RNN}(t)\| \leq R_4$ for $t \in [0,T']$, the bound on $\|e_{\xi}(t)\|$ in equation \eqref{eqn:thm4_result} in the proof of Proposition \ref{prop:RNN_easy} applies on the interval $t \in [0,T']$ under the weaker assumptions \ref{ass:UAT}, and $\|e_{\xi}(T')\| < \frac{\delta}{\rho}$. This is a contradiction. Therefore, $\sup_{t \in \mcT} \|\xi^{RNN}(t)\|\leq R_4$, and the proof of Proposition \ref{prop:RNN_easy} holds with the weaker Assumptions \ref{ass:UAT} for bounded inputs $t \in \mcT$ and $b,c \in \mathsf{Z}_R$.  
\end{proof}
The bounds on $\alpha$ and $\beta$ required in Theorem \ref{thm:RNN_approx} are justified because for known material properties $E$ and $\nu$, $\alpha$ and $\beta$ are determined and finite-dimensional, so they have maximum and minimum values. 

\section{Special Case Solutions}
\subsection{Laplace Transform Limit}\label{assec:Enu_s_limit}
Here we derive the form of $\Psi_0^{\dagger}$ in equation \eqref{eqn:homogc_general} via a power series expansion of the Laplace transform at $s = \infty$. Starting from the definition in equation \eqref{eqn:Lap_inv}: 
\begin{align*}
    \Psi_0^\dagger = \mathcal{L}^{-1}\left(\left(\int_0^1 \frac{dy}{s\nu(y) + E(y)}\right)^{-1} \p_x \widehat{u}_0\right).
\end{align*}
For $s\gg1$, we have that
\begin{equation*}
    \left(\int_0^1 \frac{dy}{s\nu(y) + E(y)}\right)^{-1} \approx \left(\int_0^1 \frac{dy}{s\nu(y)}\right)^{-1} = s\left(\frac{dy}{\nu(y)}\right)^{-1}.
\end{equation*}

Setting $\nu' = \left(\frac{dy}{\nu(y)}\right)^{-1}$, we now subtract out the linear dependence on $s$ and let $z = \frac{1}{s}$. We define 
$$F(z)=\hat{a}_0(s) -\nu' s \Big|_{s=z^{-1}}$$
to obtain
\begin{align*}
    F(z) & = \hat{a}_0\bigl(z^{-1}\bigr) - \nu' z^{-1}\\
    & = \left(\int_0^1\frac{z \; dy}{\nu(y) + z E(y)}\right)^{-1} - \left(\int_0^1 \frac{z \; dy}{\nu(y)}\right)^{-1}\\
    & = \frac{\int_0^1 \left(\frac{z}{\nu(y)} - \frac{z}{\nu(y) + zE(y)}\right)dy}{\left(\int_0^1\frac{z \; dy}{\nu(y) + z E(y)}\right)\left(\int_0^1 \frac{z \; dy}{\nu(y)}\right)}\\
    & = \frac{z^2 \int_0^1 \frac{E(y)}{\nu(y)(\nu(y)+zE(y))}\; dy}{z^2\left(\int_0^1 \frac{dy}{\nu(y)+zE(y)}\right)\left(\int_0^1 \frac{dy}{\nu(y)}\right)}\\
    & = \frac{\int_0^1 \frac{E(y)}{\nu(y)(\nu(y)+zE(y))}\; dy}{\left(\int_0^1 \frac{dy}{\nu(y)+zE(y)}\right)\left(\int_0^1 \frac{dy}{\nu(y)}\right)}.
\end{align*}
Since $\inf_{y \in (0,1)} \nu(y) > 0$, 
\begin{equation*}
    \lim_{z \to 0} F(z) = \frac{\int_0^1 \frac{E(y)}{\nu^2(y)}\;dy}{\left(\int_0^1 \frac{dy}{\nu(y)}\right)^2}=: E'
\end{equation*}
From this same computation, we see that for $\widehat{a}_0(s) = s\nu' + E' + \kappa(s)$, the contribution $\kappa(s)$ consists of lower order terms in $s$ and is such that $\lim_{s \to \infty} \kappa(s) = 0$. Using the fact that the inverse
Laplace transform of a product (if it exists) is a convolution, we justify
the form of the integral term in equation \eqref{eqn:homogc_general}.

\subsection{Forced Boundary Problem}\label{sapdx:FBP} 

\lemmaFBP*
\begin{proof}
Taking the Laplace transform of \eqref{eqns:OG} yields
\begin{equation*}
    \widehat{\sigma}(s) = (E(y) + \nu(y)s)\p_y\widehat{u}(y,s).
\end{equation*}
Spatially averaging and noting that $b(t) = \la \p_y u(y,t) \ra $, we have
\begin{equation}\label{eqn:bFBP}
    \widehat{b}(s) = \int_0^1 \frac{dy}{(E + s \nu)(y)} \widehat{\sigma}(s). 
\end{equation}
Then $\widehat{\sigma}(s) = \left( \int_0^1 \frac{dy}{(E + s \nu)(y)} \right)^{-1}\widehat{b}(s)$, which is equivalent to $\widehat{\sigma}(s) = \widehat{a_0}(s) \widehat{b}(s)$ using equation \eqref{eqn:a_0}. The definition of $\Psi_0^\dagger$ in equation \eqref{eqn:Lap_inv} completes the proof. 

\end{proof}
Lemma \ref{lem:data_jus} justifies the use of data arising from the system \eqref{eqns:OG} to train the map $\Psi_0$. 

\section{Macroscale Numeric Comparisons}
\subsection{RNN Training and Testing: Piecewise-Constant Case}\label{assec:RNN_train_test}
We trained three RNNs using the same dataset for the setting of a 2-piecewise constant material with material parameters $E_1 = 1$, $E_2 = 3$, $\nu_1 = 0.1$, and $\nu_2 = 0.2$. The data was generated using a forward Euler method with time discretization $dt = 0.001$ up to time $T = 4$ on the known analytic solution for the 2-piecewise-constant cell problem. Denote the data by $\{(\p_x u_0)_n,(\sigma_0)_n\}_{n=1}^N$ as discussed in Section \ref{ssec:RNN_Optimization}. We repeat the two loss functions here: 
\paragraph{Accessible Loss Function:}
\begin{equation*}
   L_1(\{\sigma_0\}_{n=1}^N, \{\widehat{\sigma}_0\}_{n=1}^N)  = \frac{1}{N}\sum_{n=1}^N \frac{\|(\sigma_0)_n -(\widehat{\sigma}_0)_n\|}{\|(\sigma_0)_n\|}
\end{equation*}
\paragraph{Inaccessible Loss Function:}
\begin{equation*}
   L_2(\{\sigma_0\}_{n=1}^N, \{\widehat{\sigma}_0\}_{n=1}^N, \{\xi\}_{n=1}^N, \{\widehat{\xi}_0\}_{n=1}^N)  = \frac{1}{N}\sum_{n=1}^N\left( \frac{\|(\sigma_0)_n -(\widehat{\sigma}_0)_n\|}{\|(\sigma_0)_n\|} + \frac{\|(\xi)_n -(\widehat{\xi})_n\|}{\|(\xi)_n\|_{L^2(\mathcal{D},\bbR)}}\right)
\end{equation*}

For each of the following RNNs, the architecture for $\mathcal{F}_{RNN}$ and $\mathcal{G}_{RNN}$ consists of three internal layers of SeLU units of $100$ nodes separated by linear layers, all followed by a final linear layer. The SELU function is applied element-wise as
$$ SELU(x) = s(\max(0,x) + \min(0,\alpha(\exp(x)-1))) $$
where $\alpha = 1.67326$ and $s = 1.05070$\footnote{\url{https://pytorch.org/docs/stable/generated/torch.nn.SELU.html}}.
We trained three different RNNs on the same dataset in the following manner: 
\begin{itemize}
    \item \textbf{RNN ``A"}: Using only the inaccessible loss function $L_2$, we trained on $N=400$ data points with subsampled time discretization of $dt = 0.004$ up to $T = 4$ for $1500$ epochs with a batch size of $50$. 
    \item \textbf{RNN ``B"}: First we used the inaccessible loss function $L_2$ to train on $N = 200$ data points with subsampled time discretization of $dt = 0.004$ up to $T = 2$ for $1500$ epochs with a batch size of $40$. Then we initialized a new RNN at the parameters of this RNN and trained with the accessible loss function $L_1$ for $1000$ epochs on $200$ data with batch size of $40$. 
    \item \textbf{RNN ``C"}: Using only the accessible loss function $L_1$, we trained on $N= 500$ data points with subsampled time discretization of $dt = 0.004$ up to $T = 4$ for $3000$ epochs with a batch size of $50$. 
\end{itemize}

The train and test errors are shown for the three RNNs in Figure \ref{fig:train_test}. 

\begin{figure}[ht!]
        \begin{subfigure}[b]{0.3\textwidth}
            \includegraphics[width = \textwidth]{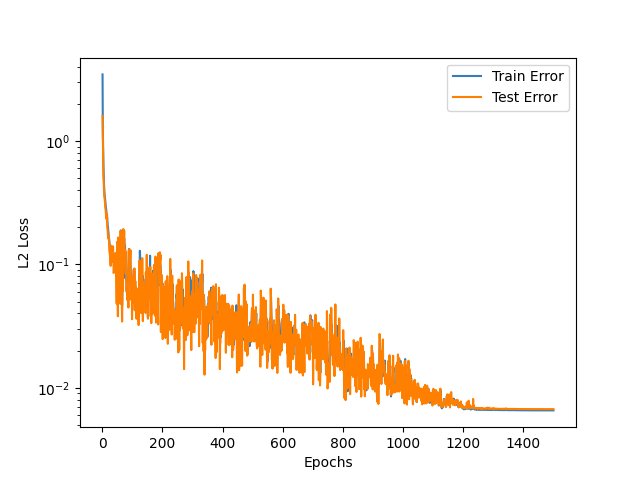}
            \caption{RNN ``A" trained using inaccessible loss function}
        \end{subfigure}
        \begin{subfigure}[b]{0.3\textwidth}
            \includegraphics[width = \textwidth]{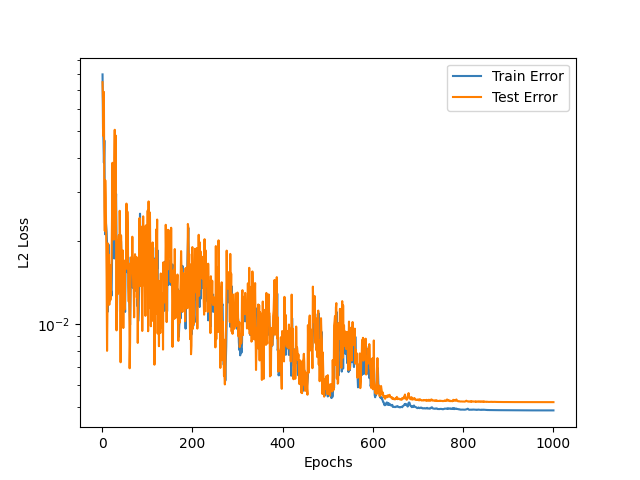}
            \caption{RNN ``B" initialized at inaccessible loss solution}
        \end{subfigure}
        \begin{subfigure}[b]{0.3\textwidth}
            \includegraphics[width = \textwidth]{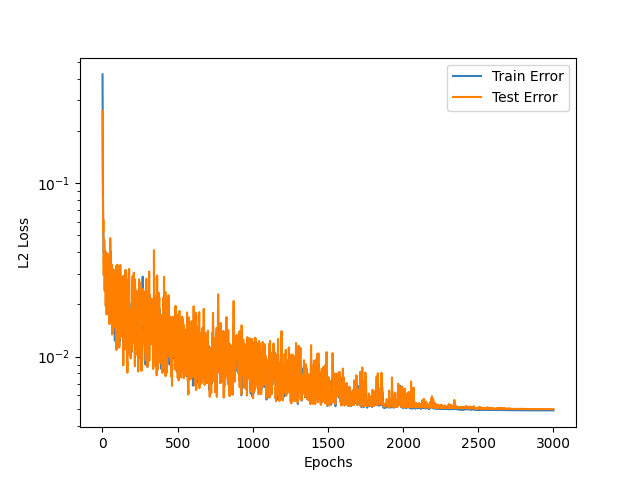}
            \caption{RNN ``C" trained using only standard loss function}
        \end{subfigure}
        \centering
        \caption{Train and test error for the three RNNs.}
        \label{fig:train_test}
\end{figure}

\subsection{RNN Training and Testing: Continuous Case}\label{assec:RNN_cont}
We trained several RNNs on data $\{(\p_xu_0)_n,(\sigma_0)_n\}_{n=1}^N$ for continuous material parameters given by $E(y) = 2 + \tanh\left(\frac{y-0.5}{0.2}\right)$ and $\nu(y) = 0.5 + 0.1\tanh\left(\frac{y-0.5}{0.2}\right)$. Each of the four RNNs had $1$, $2$, $5$, and $10$ hidden variables ($L_0$, or the dimension of $\xi$) respectively. The data was generated by solving the cell problem using
a finite difference method with $200$ spatial nodes and $dt = h^2$ where $h$ is the spatial discretization. The RNN was trained for $3000$ epochs on $500$ data. The macroscale simulations were performed with a spacial discretization of $h_{cell} = 0.05$ and a time discretization of $dt = 0.4h_{cell}^2$. They were compared to an FEM solution computed with a spacial discretization of $h = 0.004$ with a material period of $0.04$ and time discretization of $dt = 0.1h^2$. 

\section{One-Dimensional Standard Linear Solid}\label{sec:SLS}
In this section we address the model of the one-dimensional Maxwell version of the SLS, whose constitutive law depends only on the strain and strain history. The analysis  for the SLS model demonstrates that the ideas presented for the KV model extend beyond that particular setting. In Section \ref{subsec:SLS-eqns}, we present the governing equations, and in Section \ref{subsec:SLS-homog} we homogenize the system. 
\subsection{Governing Equations}\label{subsec:SLS-eqns}
In the setting without inertia,
the displacement $u_{\eps}$, strain $e_{\eps}$, and inelastic strain
$e_{\eps}^p$ are related by 
\begin{subequations}\label{eqns:SLS}
\begin{align}
-\p_x \sigma_{\eps} & = f, \\
e_{\eps} & = \p_x u,\\
    \sigma_{\eps}&= E_{1,\eps} e_{\eps}+E_{2,\eps}(e_{\eps}-e^p_{\eps}), \label{eqn:SLS1}\\
    \p_t e^p_{\eps} &=\frac{E_{2,\eps}}{\nu_{\eps}}(e_{\eps}-e^p_{\eps}); \label{eqn:SLS2}
\end{align}
\end{subequations}
where $f: \mathcal{D} \times \mathcal{T} \mapsto \bbR$ is a known forcing, and we impose initial condition $u(x,0) = 0$ and boundary conditions $u(x,t) = 0$ for $x \in \p\mathcal{D}$. We seek a solution $u_{\eps}:\mathcal{D} \times \mathcal{T} \mapsto \bbR$. Once more we have small scale dependence in the material properties through $\eps$: we have $E_{i,\eps}(x) = E_i\left(\frac{x}{\eps}\right)$ for $i = 1,2$ and $\nu_{\eps}(x) = \nu \left(\frac{x}{\eps}\right)$ for $0 < \eps \ll 1$. 
\subsection{Homogenization}\label{subsec:SLS-homog}
First, we take the Laplace transform of equation \eqref{eqns:SLS} and combine the transformed expressions of equations \eqref{eqn:SLS1} and \eqref{eqn:SLS2} 
to arrive at 
\begin{equation}
    \wh{\sigma}_{\eps} = E_{1,\eps}\wh{e}_{\eps} + E_{2,\eps}s\left(s + \frac{E_{2,\eps}}{\nu_{\eps}}\right)^{-1}\wh{e}_{\eps}
\end{equation}
Letting $\wh{a}(s) = E_{1,\eps} + E_{2,\eps}s\left(s+\frac{E_{2,\eps}}{\nu_{\eps}}\right)^{-1}$, the homogenization theory of Section \ref{subsec:homog} applies, and we can use the harmonic averaging expression in equation \eqref{eqn:a_0} to write
\begin{equation}\label{eqn:SLSa0}
    \wh{a}_0(s) = \left\la (a(s))^{-1}\right\ra^{-1} = \left(\int_0^1 \frac{s+\frac{E_{2}}{\nu}}{(E_{1}+E_{2})s + \frac{E_{1}E_{2}}{\nu}}dy\right)^{-1}
\end{equation} 
where the homogenized solution $u_0$ solves system \ref{eqns:general-dynamics-homog} with $\Psi_0^\dagger$ is defined as 
\begin{equation}\label{eqn:Lap_inv2}
\Psi_0^\dagger = \mathcal{L}^{-1}\left[\wh{a}_0(s)\p_x\wh{u}_0\right],
\end{equation}
analogous to the KV case. However, in the case of piecewise-constant $E_1$, $E_2$ and $\nu$ the inverse Laplace transform yields a different form in the SLS case: 
\begin{subequations}\label{eqn:SLS_PC_exact}
\begin{align}
    \Psi_0^{PC}(\p_x u_0, t; \theta) &= E' \p_x u_0(t) - \sum_{\ell=1}^L \xi_{\ell}(t)\\
    \p_t \xi_{\ell}(t) & = \beta_{\ell}\p_x u_0(t) - \alpha_{\ell}\xi_{\ell}(t), \quad \ell \in \{1, \dots, L\}
\end{align}
for a material with $L$ pieces. Note that this model does not have dependence on the strain rate, but it has one more internal variable than the piecewise-constant case for the KV model. 
The value of $E'$ follows from taking the limit $s \to \infty$ and is given by
\begin{equation*}    E'=\left(\int_0^1 \frac{1}{(E_{1}+E_{2})}dy\right)^{-1}.
\end{equation*}

Full derivation may be found in Appendix \ref{assec:SLS}.
\end{subequations}

\subsection{SLS Derivation}\label{assec:SLS}
Here we show that the SLS model has one more internal variable in the piecewise-constant case than the KV model does. This is the analog of Theorem \ref{thm:PC_exact} for the SLS model. 
Starting from equation \eqref{eqn:SLSa0} for $\hat{a}_0(s)$:
\begin{align*}
    \widehat{a}_0(s) & = \la \widehat{a}(s)^{-1}\ra^{-1} \\
    & = \left(\int_0^1\frac{s + \frac{E_2(y)}{\nu}}{(E_1(y)  + E_2(y))s + \frac{E_1(y)E_2(y)}{\nu(y)}}\;dy \right)^{-1}\\
    & \left(\sum_{i=1}^L \frac{(s + \frac{E_{2,i}}{\nu_i})d_i}{(E_{1,i} + E_{2,i})s + \frac{E_{1,i}E_{2,i}}{\nu_i}}\right)^{-1}
\end{align*}
for $L$-piecewise-constant $E_1$, $E_2$, and $\nu$ with pieces of length $d_i$. Let $c_i = \frac{E_{2,i}}{\nu_i}$, $k_i = E_{1,i} + E_{2,i}$ and $p_i = \frac{E_{1,i}E_{2,i}}{\nu_i}$. Continuing, \begin{align*}
    & = \left(\sum_{i=1}^L \frac{(s+c_i)d_i}{k_i s + p_i}\right)^{-1}\\
    & = \frac{\prod_{i=1}^L (k_i s + p_i)}{\sum_{i=1}^L d_i (s+c_i)\prod_{j\neq i}(k_j s + p_j)}:= \frac{P(s)}{Q(s)}
\end{align*}
Note that both $P(s)$ and $Q(s)$ have degree $L$. There is a unique constant $E'$ such that 
\begin{equation*}
    \frac{P(s)}{Q(s)} = E' + \frac{C(s)}{Q(s)}
\end{equation*}
where $C(s)$ has degree $L$. Then $\frac{C(s)}{Q(s)}$ decomposes uniquely as $\sum_{\ell = 1}^L \frac{\beta_{\ell}}{s + \alpha_{\ell}}$. Note that this is one more pole than the decomposition for the KV model in Theorem \ref{thm:PC_exact} has. We will now show that roots of $Q$ are real and negative, which will lead to the expression in equation \eqref{eqn:SLS_PC_exact}. First notice that if the roots of $Q(s)$ are real, then they must be negative since $k_i$, $c_i$, $d_i$, and $p_i$ are strictly positive for all $i \in [L]$. Suppose for the sake of contradiction that $Q(s)$ has a root with a nonzero imaginary component: $s = a+bi$ where $b \neq 0$. Then
\begin{align*}
    Q(a+bi) & = \sum_{\ell = 1}^L d_i (a+bi +c_{\ell}) \prod_{j \neq \ell}(k_j(a+bi)+p_j)\\
    & = \left(\prod_j(k_j(a+bi)+p_j)\right)\sum_{\ell = 1}^L \frac{d_{\ell}(a + bi+c_{\ell})}{k_{\ell}(a+bi) + p_{\ell}}\\
    & = \left(\prod_j(k_j(a+bi)+p_j)\right)\sum_{\ell = 1}^L \left(\frac{d_{\ell}a + d_{\ell}c_{\ell}+d_{\ell}bi}{k_{\ell}a + p_{\ell} + k_{\ell}bi}\right)\left(\frac{k_{\ell}a + p_{\ell}-k_{\ell}bi}{k_{\ell}a+p_{\ell}-k_{\ell}bi}\right)\\
    & =\left(\prod_j(k_j(a+bi)+p_j)\right)\times \\ &\sum_{\ell = 1}^L \frac{d_{\ell}}{(k_{\ell}a+p_{\ell})^2 + (k_{\ell}b)^2}\left[\left((a +c_{\ell})(k_{\ell}a+p_{\ell}) + k_{\ell}b^2\right) + \left(-k_{\ell}bc_{\ell}+ bp_{\ell}\right)i\right]
\end{align*}
If $a+bi$ is a root of $Q$, then we need $b\sum_{\ell =1}^L \frac{d_{\ell}}{(k_{\ell}a+p_{\ell})^2 + (k_{\ell}b)^2}(-k_{\ell}c_{\ell} +p_{\ell})  = 0$. Notice that $-k_{\ell}c_{\ell} + p_{\ell} = -\frac{E_{2,\ell}^2}{\nu_{\ell}}$, which is strictly negative, so for $b \neq 0$, $\Im(Q(a+bi))<0$, which is a contradiction. Therefore, $b =0$, and all the roots of $Q$ are real and negative. Inverting the Laplace transform, we arrive at equation \eqref{eqn:SLS_PC_exact}.

\end{document}